\documentclass[11pt,reqno,final,letterpaper,oneside]{amsart}

\usepackage{amsmath}
\usepackage{amssymb}
\usepackage{bm}
\usepackage{mathrsfs}
\usepackage[dvips]{graphicx}

\usepackage{color}

\usepackage{url}

\usepackage{supertabular}

\newtheorem{thm}{Theorem}
\newtheorem{cor}[thm]{Corollary}

\newtheorem{prop}[thm]{Proposition}

\newtheorem{algorithm}[thm]{Algorithm}

\numberwithin{equation}{section}

\DeclareMathAlphabet{\mathsfsl}{OT1}{cmss}{m}{sl}

\newcommand{\lang}{\textit}

\newcommand{\term}{\emph}

\newenvironment{inputs}%
	{\makebox{\phantom{}} \\ %
		\textsc{Input:}\begin{itemize}}%
	{\end{itemize}}

\newenvironment{outputs}%
	{\textsc{Output:}\begin{itemize}}%
	{\end{itemize}}

\newenvironment{procedure}%
	{\textsc{Procedure:}\begin{enumerate}}%
	{\end{enumerate}}

\renewcommand{\phi}{\varphi}

\newcommand{\defby}{\overset{\mathrm{\scriptscriptstyle{def}}}{=}}
\newcommand{\half}{\tfrac{1}{2}}

\newcommand{\onevct}{\mathbf{e}}

\newcommand{\Id}{\mathbf{I}}

\newcommand{\coll}[1]{\mathscr{#1}}
\newcommand{\subsp}[1]{\mathcal{#1}}

\newcommand{\Sspace}[1]{\mathbb{S}^{#1}}

\newcommand{\Rspace}[1]{\mathbb{R}^{#1}}
\newcommand{\Cspace}[1]{\mathbb{C}^{#1}}

\newcommand{\RPspace}[1]{\mathbb{P}^{#1}(\mathbb{R})}
\newcommand{\CPspace}[1]{\mathbb{P}^{#1}(\mathbb{C})}
\newcommand{\FPspace}[1]{\mathbb{P}^{#1}(\mathbb{F})}

\newcommand{\RGspace}[2]{\mathbb{G}( {#1}, \Rspace{#2} )}
\newcommand{\CGspace}[2]{\mathbb{G}( {#1}, \Cspace{#2} )}
\newcommand{\FGspace}[2]{\mathbb{G}( {#1}, \mathbb{F}^{#2} )}

\newcommand{\grp}[1]{\mathbb{#1}}

\newcommand{\abs}[1]{\left\vert {#1} \right\vert}

\newcommand{\vct}[1]{\bm{#1}}
\newcommand{\mtx}[1]{\bm{#1}}

\newcommand{\adj}{*}

\newcommand{\range}{\operatorname{range}}

\newcommand{\rank}{\operatorname{rank}}

\newcommand{\diag}{\operatorname{diag}}
\newcommand{\trace}{\operatorname{trace}}

\newcommand{\psd}{\succcurlyeq}

\newcommand{\ip}[2]{\left\langle {#1},\ {#2} \right\rangle}

\newcommand{\norm}[1]{\left\Vert {#1} \right\Vert}

\newcommand{\smnorm}[2]{{\bigl\Vert {#2} \bigr\Vert}_{#1}}

\newcommand{\enorm}[1]{\norm{#1}_2}
\newcommand{\enormsq}[1]{\enorm{#1}^2}

\newcommand{\fnorm}[1]{\norm{#1}_{\mathrm{F}}}
\newcommand{\fnormsq}[1]{\fnorm{#1}^2}

\newcommand{\pnorm}[2]{\norm{#2}_{#1}}

\newcommand{\dist}{\operatorname{dist}}

\newcommand{\pack}{\operatorname{pack}}

\newcommand{\subjto}{\quad\text{subject to}\quad}

\newcommand{\chord}{\mathrm{chord}}
\newcommand{\spec}{\mathrm{spec}}
\newcommand{\fs}{\mathrm{FS}}
\newcommand{\geod}{\mathrm{geo}}

\begin{document}

\title[Constructing Grassmannian Packings]{Constructing Packings in Grassmannian Manifolds \\ via Alternating Projection}

\author[Dhillon et al.]{I.\ S.\ Dhillon, R.\ W.\ Heath Jr., T.\ Strohmer, and J.\ A.\ Tropp}%
%	\thanks{Department of Computer Sciences, University of Texas, Austin, TX 78712.  E-mail: \protect\url{inderjit@cs.utexas.edu}. Supported by NSF grant CCF-0431257, NSF Career Award ACI-0093404, and NSF-ITR award IIS-0325116},
%R.\ W.\ Heath Jr.%
%	\thanks{Department of Electrical and Computer Engineering, University of Texas, Austin, TX 78712.  E-mail: \protect\url{rheath@ece.utexas.edu}.  Supported by NSF CCF Grant \#514194.},
%T.\ Strohmer%
%	\thanks{Department of Mathematics, University of California, Davis, CA 95616.  E-mail: \protect{\url{strohmer@math.ucdavis.edu}}.  Supported by NSF DMS Grant \#0511461.},
%and J.\ A.\ Tropp%
%	\thanks{Department of Mathematics, University of Michigan, Ann Arbor, MI 48109.  E-mail: \protect{\url{jtropp@umich.edu}}.  Supported by an NSF Graduate Fellowship and a J.\ T.\ Oden Visiting Faculty Fellowship.}.}

\thanks{{\it E-mail.} \texttt{inderjit@cs.utexas.edu}, \texttt{rheath@ece.utexas.edu}, \texttt{strohmer@ucdavis.edu}, \texttt{jtropp@acm.caltech.edu}}

\thanks{{\it Addresses.} ISD is with the Department of Computer Sciences, University of Texas, Austin, TX 78712.  RWH is with the Department of Electrical and Computer Engineering, University of Texas, Austin, TX 78712.  TS is with the Department of Mathematics, University of California, Davis, CA 95616.  JAT is currently with Applied and Computational Mathematics, California Institute of Technology, Pasadena, CA 91125.}

\thanks{{\it Acknowledgments.} ISD was supported by NSF grant CCF-0431257, NSF Career Award ACI-0093404, and NSF-ITR award IIS-0325116.  RWH was supported by NSF CCF Grant \#514194.  TS was supported by NSF DMS Grant \#0511461.  JAT was supported by an NSF Graduate Fellowship, a J.\ T.\ Oden Visiting Faculty Fellowship, and NSF DMS 0503299.}

\date{May 2004.  Revised November 2006 and August 2007.}

\begin{abstract}
This paper describes a numerical method for finding good packings in Grassmannian manifolds equipped with various metrics.  This investigation also encompasses packing in projective spaces.  In each case, producing a good packing is equivalent to constructing a matrix that has certain structural and spectral properties.  By alternately enforcing the structural condition and then the spectral condition, it is often possible to reach a matrix that satisfies both.  One may then extract a packing from this matrix.

This approach is both powerful and versatile.  In cases where experiments have been performed, the alternating projection method yields packings that compete with the best packings recorded.  It also extends to problems that have not been studied numerically.  For example, it can be used to produce packings of subspaces in real and complex Grassmannian spaces equipped with the Fubini--Study distance; these packings are valuable in wireless communications.  One can prove that some of the novel configurations constructed by the algorithm have packing diameters that are nearly optimal.
\end{abstract}



\maketitle

%%%%%%%%%%
%%%%%%%%%%

\section{Introduction}

Let us begin with the standard facetious example.  Imagine that several mutually inimical nations build their capital cities on the surface of a featureless globe.  Being concerned about missile strikes, they wish to locate the closest pair of cities as far apart as possible.  In other words, what is the best way to pack points on the surface of a two-dimensional sphere?

This question, first discussed by the Dutch biologist Tammes \cite{Tam30:Origin-Number}, is the prototypical example of packing in a compact metric space.  It has been studied in detail for the last 75 years.  More recently, researchers have started to ask about packings in other compact spaces.  In particular, several communities have investigated how to arrange subspaces in a Euclidean space so that they are as distinct as possible.  An equivalent formulation is to find the best packings of points in a Grassmannian manifold.  This problem has applications in quantum computing and wireless communications.  There has been theoretical interest in subspace packing since the 1960s \cite{Tot65:Distribution-Points}, but the first detailed numerical study appears in a 1996 paper of Conway, Hardin, and Sloane \cite{CHS96:Packing-Lines}. 

The aim of this paper is to describe a flexible numerical method that can be used to construct packings in Grassmannian manifolds equipped with several different metrics.  The rest of this introduction provides a formal statement of abstract packing problems, and it offers an overview of our approach to solving them.

\subsection{Abstract Packing Problems}

Although we will be working with Grassmannian manifolds, it is more instructive to introduce packing problems in an abstract setting.  Let $\mathbb{M}$ be a compact metric space endowed with the distance function $\dist_{\mathbb{M}}$.  The \term{packing diameter} of a finite subset $\coll{X}$ is the minimum distance between some pair of distinct points drawn from $\coll{X}$.  That is,
$$
\pack_{\mathbb{M}}(\coll{X}) \defby
\min_{m \neq n} \dist_{\mathbb{M}}( x_m, x_n ).
$$
In other words, the packing diameter of a set is the diameter of the largest open ball that can be centered at each point of the set without encompassing any other point.  (It is also common to study the \term{packing radius}, which is half the diameter of this ball.)  An optimal packing of $N$ points is an ensemble $\coll{X}$ that solves the mathematical program
$$
\max_{\abs{\coll{X}} = N} \pack_{\mathbb{M}}(\coll{X})
$$
where $\abs{\cdot}$ returns the cardinality of a finite set.  The optimal packing problem is guaranteed to have a solution because the metric space is compact and the objective is a continuous function of the ensemble $\coll{X}$.

This article focuses on a \term{feasibility problem} closely connected with optimal packing.  Given a number $\rho$, the goal is to produce a set of $N$ points for which
\begin{equation} \label{eqn:feasibility}
\pack_{\mathbb{M}}(\coll{X}) \geq \rho.
\end{equation}
This problem is notoriously difficult to solve because it is highly nonconvex, and it is even more difficult to determine the maximum value of $\rho$ for which the feasibility problem is soluble.  This maximum value of $\rho$ corresponds with the diameter of an optimal packing.

\subsection{Alternating Projection}

We will attempt to solve the feasibility problem \eqref{eqn:feasibility} in Grassmannian manifolds equipped with a number of different metrics, but the same basic algorithm applies in each case.  Here is a high-level description of our approach.

First, we show that each configuration of subspaces is associated with a block Gram matrix whose blocks control the distances between pairs of subspaces.  Then we prove that a configuration solves the feasibility problem \eqref{eqn:feasibility} if and only if its Gram matrix possesses both a structural property and a spectral property.  The overall algorithm consists of the following steps.

\begin{enumerate}
\item	Choose an initial configuration and construct its matrix.

\item	Alternately enforce the structural condition and the spectral condition in hope of reaching a matrix that satisfies both.

\item	Extract a configuration of subspaces from the output matrix.
\end{enumerate}

In our work, we choose the initial configuration randomly and then remove similar subspaces from it with a simple algorithm.  One can imagine more sophisticated approaches to constructing the initial configuration.

Flexibility and ease of implementation are the major advantages of alternating projection.  This article demonstrates that appropriate modifications of this basic technique allow us to construct solutions to the feasibility problem in Grassmannian manifolds equipped with various metrics.  Some of these problems have never been studied numerically, and the experiments point toward intriguing phenomena that deserve theoretical attention.  Moreover, we believe that the possibilities of this method have not been exhausted and that it will see other applications in the future.

Alternating projection does have several drawbacks.  It may converge very slowly, and it does not always yield a high level of numerical precision.  In addition, it may not deliver good packings when the ambient dimension or the number of subspaces in the configuration is large.

\subsection{Motivation and Related Work}

This work was motivated by applications in electrical engineering.  In particular, subspace packings solve certain extremal problems that arise in multiple-antenna communication systems \cite{ZheTse:Communication-on-the-Grassmann-manifold::02,HocMarRic:Systematic-design-of-unitary:00,LovHeaSan:What-is-the-value-of-limited:04}.  This application requires complex Grassmannian packings that consist of a small number of subspaces in an ambient space of low dimension.  Our algorithm is quite effective in this parameter regime.  The resulting packings fill a significant gap in the literature, since existing tables consider only the real case \cite{Slo:Grassmannian-Web}.  See Section \ref{sec:discussion} for additional discussion of the wireless application.

The approach to packing via alternating projection was discussed in a previous publication \cite{TDHS05:Designing-Structured}, but the experiments were limited to a single case.  We are aware of several other numerical methods that can be used to construct packings in Grassmannian manifolds \cite{CHS96:Packing-Lines,Tro99:Approximate-Maximin,ARU01:Multiple-Antenna-Signal}.  These techniques rely on ideas from nonlinear programming.

%\notate{RWH.  Mention beamforming approaches here.}

\subsection{Historical Interlude}

The problem of constructing optimal packings in various metric spaces has a long and lovely history.  The most famous example may be Kepler's Conjecture that an optimal packing of spheres in three-dimensional Euclidean space%
\footnote{The infinite extent of a Euclidean space necessitates a more subtle definition of an optimal packing.}
locates them at the points of a face-centered cubic lattice.  For millennia, greengrocers have applied this theorem when stacking oranges, but it has only been established rigorously within the last few years \cite{Hal04:Proof-Kepler}.  Packing problems play a major role in modern communications because error-correcting codes may be interpreted as packings in the Hamming space of binary strings \cite{CT91:Elements-Information}.  The standard reference on packing is the \lang{magnum opus} of Conway and Sloane \cite{CS98:Sphere-Packing}.  Classical monographs on the subject were written by L.\ Fejes T\'oth \cite{FT64:Regular-Figures} and C.\ A.\ Rogers \cite{Rog64:Packing-Covering}.

The idea of applying alternating projection to feasibility problems first appeared in the work of von Neumann \cite{vN50:Functional-Operators}.  He proved that an alternating projection between two closed subspaces of a Hilbert space converges to the orthogonal projection of the initial iterate onto the intersection of the two subspaces.  Cheney and Goldstein subsequently showed that an alternating projection between two closed, convex subsets of a Hilbert space always converges to a point in their intersection (provided that the intersection is nonempty) \cite{CG59:Proximity-Maps}.  This result does not apply in our setting because one of the constraint sets we define is not convex.

%  Unfortunately, this result does not apply to our problem because the constraint set is not convex.

\subsection{Outline of Article}

Here is a brief overview of this article.  In Section \ref{sec:grass}, we develop a basic description of Grassmannian manifolds and present some natural metrics.  Section \ref{sec:alt-proj} explains why alternating projection is a natural algorithm for producing Grassmannian packings, and it outlines how to apply this algorithm for one specific metric.  Section \ref{sec:rankin} gives some theoretical upper bounds on the optimal diameter of packings in Grassmannian manifolds.  Section \ref{sec:exp} describes the outcomes of an extensive set of numerical experiments and explains how to apply the algorithm to other metrics.  Section \ref{sec:open} offers some discussion and conclusions.  Appendix \ref{app:tammes} explores how our methodology applies to Tammes' Problem of packing on the surface of a sphere.  Finally, Appendix \ref{app:tables} contains tables and figures that detail the experimental results.

\section{Packing in Grassmannian Manifolds} \label{sec:grass}

This section introduces our notation and a simple description of the Grassmannian manifold.  It presents several natural metrics on the manifold, and it shows how to represent a configuration of subspaces in matrix form.

\subsection{Preliminaries}

We work in the vector space $\Cspace{d}$.  The symbol ${}^*$ denotes the complex-conjugate transpose of a vector (or matrix).  We equip the vector space with its usual inner product $\ip{ \vct{x} }{ \vct{y} } = \vct{y}^\adj \vct{x}$.  This inner product generates the $\ell_2$ norm via the formula $\enormsq{ \vct{x} } = \ip{ \vct{x} }{ \vct{x} }$.

The $d$-dimensional identity matrix is $\Id_d$; we sometimes omit the subscript if it is unnecessary.  A square matrix is \term{positive semidefinite} when its eigenvalues are all nonnegative.  We write $\mtx{X} \psd \mtx{0}$ to indicate that $\mtx{X}$ is positive semidefinite.

A square, complex matrix $\mtx{U}$ is \term{unitary} if it satisfies $\mtx{U}^\adj \mtx{U} = \Id$.  If in addition the entries of $\mtx{U}$ are real, the matrix is \term{orthogonal}.  The unitary group $\grp{U}(d)$ can be presented as the collection of all $d \times d$ unitary matrices with ordinary matrix multiplication.  The real orthogonal group $\grp{O}(d)$ can be presented as the collection of all $d \times d$ real orthogonal matrices with the usual matrix multiplication.

Suppose that $\mtx{X}$ is a general matrix.  The Frobenius norm is calculated as $\fnormsq{ \mtx{X} } = \trace{ \mtx{X}^\adj \mtx{X} }$, where the trace operator sums the diagonal entries of the matrix.  The spectral norm is denoted by $\pnorm{2,2}{ \mtx{X} }$; it returns the largest singular value of $\mtx{X}$.  Both these norms are unitarily invariant, which means that $\norm{ \mtx{UXV}^\adj } = \norm{ \mtx{X} }$ whenever $\mtx{U}$ and $\mtx{V}$ are unitary.

%the vector space
%the inner product and the conjugate transpose operator
%the norms: euclidean, frobenius, spectral, unitary invariance
%the trace and determinant

%the identity matrix

%
%positive definite matrices and the partial ordering.

\subsection{Grassmannian Manifolds}

The (complex) Grassmannian manifold $\CGspace{K}{d}$ is the collection of all $K$-dimensional subspaces of $\Cspace{d}$.  This space is isomorphic to a quotient of unitary groups:
$$
\CGspace{K}{d} \cong
\frac{ \grp{U}(d) }{ \grp{U}(K) \times \grp{U}(d - K) }.
$$
To understand the equivalence, note that each orthonormal basis from $\Cspace{d}$ can be split into $K$ vectors, which span a $K$-dimensional subspace, and $d - K$ vectors, which span the orthogonal complement of that subspace.  To obtain a unique representation for the subspace, it is necessary to divide by isometries that fix the subspace and by isometries that fix its complement.  It is evident that $\CGspace{K}{d}$ is always isomorphic to $\CGspace{d - K}{d}$.

Similarly, the real Grassmannian manifold $\RGspace{K}{d}$ is the collection of all $K$-dimensional subspaces of $\Rspace{d}$.  This space is isomorphic to a quotient of orthogonal groups:
$$
\RGspace{K}{d} \cong
\frac{ \grp{O}(d) }{ \grp{O}(K) \times \grp{O}(d - K) }.
$$
If we need to refer to the real and complex Grassmannians simultaneously, we write $\FGspace{K}{d}$.

In the theoretical development, we concentrate on complex Grassmannians since the development for the real case is identical, except that all the matrices are real-valued instead of complex-valued.  A second reason for focusing on the complex case is that complex packings arise naturally in wireless communications \cite{LHS03:Grassmannian-Beamforming}.

When each subspace has dimension $K = 1$, the Grassmannian manifold reduces to a simpler object called a \term{projective space}.  The elements of a projective space can be viewed as lines through the origin of a Euclidean space.   The standard notation is $\FPspace{d-1} \defby \FGspace{1}{d}$.  We will spend a significant amount of attention on packings of this manifold.

\subsection{Principal Angles}

Suppose that $\subsp{S}$ and $\subsp{T}$ are two subspaces in $\CGspace{K}{d}$.  These subspaces are inclined against each other by $K$ different \term{principal angles}.  The smallest principal angle $\theta_1$ is the minimum angle formed by a pair of unit vectors $(\vct{s}_1, \vct{t}_1)$ drawn from $\subsp{S} \times \subsp{T}$.  That is,
$$
\theta_1 = \min_{(\vct{s}_1, \vct{t}_1) \in \subsp{S} \times \subsp{T}}
	\arccos \ip{ \vct{s}_1 }{ \vct{t}_1 }
\subjto
	\enorm{ \vct{s}_1 } = 1 \quad\text{and}\quad
	\enorm{ \vct{t}_1 } = 1.
$$
The second principal angle $\theta_2$ is defined as the smallest angle attained by a pair of unit vectors $(\vct{s}_2, \vct{t}_2)$ that is orthogonal to the first pair, i.e.,
\begin{align*}
\theta_2 = \min_{(\vct{s}_2, \vct{t}_2) \in \subsp{S} \times \subsp{T}}
	\arccos \ip{ \vct{s}_2 }{ \vct{t}_2 }
\subjto
	&\enorm{ \vct{s}_2 } = 1 \quad\text{and}\quad
	 \enorm{ \vct{t}_2 } = 1, \\
&\ip{ \vct{s}_1 }{ \vct{s}_2 } = 0
\quad\text{and}\quad
\ip{ \vct{t}_1 }{ \vct{t}_2 } = 0.
\end{align*}
The remaining principal angles are defined analogously.  The sequence of principal angles is nondecreasing, and it is contained in the range $[0, \pi/2]$.  We only consider metrics that are functions of the principal angles between two subspaces.

Let us present a more computational definition of the principal angles \cite{BG73:Numerical-Methods}.  Suppose that the columns of $\mtx{S}$ and $\mtx{T}$ form orthonormal bases for the subspaces $\subsp{S}$ and $\subsp{T}$.  More rigorously, $\mtx{S}$ is a $d \times K$ matrix that satisfies $\mtx{S}^\adj \mtx{S} = \Id_K$ and $\range{\mtx{S}} = \subsp{S}$.  The matrix $\mtx{T}$ has an analogous definition.  Next we compute a singular value decomposition of the product $\mtx{S}^\adj \mtx{T}$:
$$
\mtx{S}^\adj \mtx{T} = \mtx{U}\mtx{C}\mtx{V}^\adj
$$
where $\mtx{U}$ and $\mtx{V}$ are $K \times K$ unitary matrices and $\mtx{C}$ is a nonnegative, diagonal matrix with nonincreasing entries.  The matrix $\mtx{C}$ of singular values is uniquely determined, and its entries are the cosines of the principal angles between $\subsp{S}$ and $\subsp{T}$:
$$
c_{kk} = \cos \theta_k
\qquad	k = 1,2, \dots, K.
$$
This definition of the principal angles is most convenient numerically because singular value decompositions can be computed efficiently with standard software.  We also note that this definition of the principal angles does not depend on the choice of matrices $\mtx{S}$ and $\mtx{T}$ that represent the two subspaces.

\subsection{Metrics on Grassmannian Manifolds}

Grassmannian manifolds admit many interesting metrics, which lead to different packing problems.  This section describes some of these metrics.

\begin{enumerate}
\item	The \term{chordal distance} between two $K$-dimensional subspaces $\subsp{S}$ and $\subsp{T}$ is given by
\begin{align} \label{eqn:chordal-distance}
\dist_{\chord}(\subsp{S}, \subsp{T})
&\defby \sqrt{ \sin^2 \theta_1 + \dots + \sin^2 \theta_K } \notag \\
& = \left[ K - \fnormsq{\mtx{S}^\adj \mtx{T}} \right]^{1/2}.
\end{align}
The values of this metric range between zero and $\sqrt{K}$.  The chordal distance is the easiest to work with, and it also yields the most symmetric packings \cite{CHS96:Packing-Lines}.

\item	The \term{spectral distance} is
\begin{align} \label{eqn:spectral-distance}
\dist_{\spec}(\subsp{S}, \subsp{T})
&\defby \min\nolimits_k \sin \theta_k \notag \\
& = \left[ 1 - \pnorm{2,2}{\mtx{S}^\adj \mtx{T}}^2 \right]^{1/2}.
\end{align}
The values of this metric range between zero and one.  As we will see, this metric promotes a special type of packing called an \term{equi-isoclinic} configuration of subspaces.

\item	The \term{Fubini--Study distance} is
\begin{align} \label{eqn:fs-distance}
\dist_{\fs}(\subsp{S}, \subsp{T})
&\defby \arccos \left(\prod\nolimits_k \cos \theta_k \right) \notag \\
& = \arccos \abs{ \det \mtx{S}^\adj \mtx{T} }.
\end{align}
This metric takes values between zero and $\pi / 2$.  It plays an important role in wireless communications \cite{LovHea:Limited-feedback-unitary-stbc:05,LovHea:Limited-feedback-unitary:05}.

%The Fubini--Study distance is the unique Riemannian metric invariant under actions of the unitary group on the Grassmannian manifold.

\item	The \term{geodesic distance} is
\begin{align*}
\dist_{\geod}(\subsp{S}, \subsp{T})
&\defby \sqrt{ \theta_1^2 + \dots + \theta_K^2 }.
\end{align*}
This metric takes values between zero and $\pi \sqrt{K}  / 2$.  From the point of view of differential geometry, the geodesic distance is very natural, but it does not seem to lead to very interesting packings \cite{CHS96:Packing-Lines}, so we will not discuss it any further.
\end{enumerate}

Grassmannian manifolds support several other interesting metrics, some of which are listed in \cite{BN02:Bounds-Packings}.  In case we are working in a projective space, i.e., $K = 1$, all of these metrics reduce to the acute angle between two lines or the sine thereof.  Therefore, the metrics are equivalent up to a monotonically increasing transformation, and they promote identical packings.

\subsection{Representing Configurations of Subspaces}

Suppose that $\coll{X} = \{ \mathcal{S}_1, \dots, \mathcal{S}_N \}$ is a collection of $N$ subspaces in $\CGspace{K}{d}$.  Let us develop a method for representing this configuration numerically.  To each subspace $\subsp{S}_n$, we associate a (nonunique) $d \times K$ matrix $\mtx{X}_n$ whose columns form an orthonormal basis for that subspace, i.e., $\mtx{X}_n^\adj \mtx{X}_n = \Id_K$ and $\range \mtx{X}_n = \subsp{S}_n$.  Now collate these $N$ matrices into a $d \times KN$ configuration matrix
$$
\mtx{X} \defby \begin{bmatrix}
\mtx{X}_1 & \mtx{X}_2 & \dots & \mtx{X}_N
\end{bmatrix}.
$$
In the sequel, we do not distinguish between the configuration $\coll{X}$ and the matrix $\mtx{X}$.

The \term{Gram matrix} of $\mtx{X}$ is defined as the $KN \times KN$ matrix $\mtx{G} = \mtx{X}^\adj \mtx{X}$.  By construction, the Gram matrix is positive semidefinite, and its rank does not exceed $d$.  It is best to regard the Gram matrix as an $N \times N$ block matrix comprised of $K \times K$ blocks, and we index it as such.  Observe that each block satisfies
$$
\mtx{G}_{mn} = \mtx{X}_m^\adj \mtx{X}_n.
$$
In particular, each diagonal block $\mtx{G}_{nn}$ is an identity matrix.  Meanwhile, the singular values of the off-diagonal block $\mtx{G}_{mn}$ equal the cosines of the principal angles between the two subspaces $\range{\mtx{X}_{m}}$ and $\range{\mtx{X}_{n}}$.

Conversely, let $\mtx{G}$ be an $N \times N$ block matrix with each block of size $K \times K$.  Suppose that the matrix is positive semidefinite, that its rank does not exceed $d$, and that its diagonal blocks are identity matrices.  Then we can factor $\mtx{G} = \mtx{X}^\adj \mtx{X}$ where $\mtx{X}$ is a $d \times KN$ configuration matrix.  That is, the columns of $\mtx{X}$ form orthogonal bases for $N$ different $K$-dimensional subspaces of $\Cspace{d}$.

As we will see, each metric on the Grassmannian manifold leads to a measure of ``magnitude'' for the off-diagonal blocks on the Gram matrix $\mtx{G}$.  A configuration solves the feasibility problem \eqref{eqn:feasibility} if and only if each off-diagonal block of its Gram matrix has sufficiently small magnitude.  So solving the feasibility problem is equivalent to producing a Gram matrix with appropriate properties.

\section{Alternating Projection for Chordal Distance} \label{sec:alt-proj}

In this section, we elaborate on the idea that solving the feasibility problem is equivalent with constructing a Gram matrix that meets certain conditions.  These conditions fall into two different categories: structural properties and spectral properties.  This observation leads naturally to an alternating projection algorithm for solving the feasibility problem.  The algorithm alternately enforces the structural properties and then the spectral properties in hope of producing a Gram matrix that satisfies them all.  This section illustrates how this approach unfolds when distances are measured with respect to the chordal metric.  In Section \ref{sec:exp}, we describe adaptations for other metrics.

\subsection{Packings with Chordal Distance}

Suppose that we seek a packing of $N$ subspaces in $\CGspace{K}{d}$ equipped with the chordal distance.  If $\mtx{X}$ is a configuration of $N$ subspaces, its packing diameter is
\begin{align*}
\pack_{\chord}(\mtx{X})
&\defby \min_{m \neq n} \dist_{\chord}(\mtx{X}_m, \mtx{X}_n) \\
& = \min_{m \neq n} \left[ K - \fnormsq{\mtx{X}_m^\adj \mtx{X}_n} \right]^{1/2}.
\end{align*}
Given a parameter $\rho$, the feasibility problem elicits a configuration $\mtx{X}$ that satisfies
$$
\min_{m \neq n} \left[ K - \fnormsq{ \mtx{X}_m^\adj \mtx{X}_n } \right]^{1/2}
	\geq \rho.
$$
We may rearrange this inequality to obtain a simpler condition:
\begin{equation} \label{eqn:chord-feas}
\max_{m \neq n} \fnorm{ \mtx{X}_m^\adj \mtx{X}_n }
	\leq \mu
\end{equation}
where
\begin{equation} \label{eqn:feas-parm}
\mu = \sqrt{K - \rho^2}.
\end{equation}

In fact, we may formulate the feasibility problem purely in terms of the Gram matrix.  Suppose that the configuration $\mtx{X}$ satisfies \eqref{eqn:chord-feas} with parameter $\mu$.  Then its Gram matrix $\mtx{G}$ must have the following six properties:
\begin{enumerate}
\item	$\mtx{G}$ is Hermitian.
\item	Each diagonal block of $\mtx{G}$ is an identity matrix.
\item	$\fnorm{ \mtx{G}_{mn} } \leq \mu$ for each $m \neq n$.
\item	$\mtx{G}$ is positive semidefinite.
\item	$\mtx{G}$ has rank $d$ or less.
\item	$\mtx{G}$ has trace $KN$.
\end{enumerate}
Some of these properties are redundant, but we have listed them separately for reasons soon to become apparent.  Conversely, suppose that a matrix $\mtx{G}$ satisfies Properties 1--6.  Then it is always possible to factor it to extract a configuration of $N$ subspaces that solves \eqref{eqn:chord-feas}.  The factorization of $\mtx{G} = \mtx{X}^\adj \mtx{X}$ can be obtained most easily from an eigenvalue decomposition of $\mtx{G}$.

\subsection{The Algorithm}

Observe that Properties 1--3 are \emph{structural} properties.  By this, we mean that they constrain the entries of the Gram matrix directly.  Properties 4--6, on the other hand, are \emph{spectral} properties.  That is, they control the eigenvalues of the matrix.  It is not easy to enforce structural and spectral properties simultaneously, so we must resort to half measures.  Starting from an initial matrix, our algorithm will alternately enforce Properties 1--3 and then Properties 4--6 in hope of reaching a matrix that satisfies all six properties at once.

To be more rigorous, let us define the structural constraint set
\begin{multline} \label{eqn:H-chord}
\coll{H}(\mu) \defby
\{ \mtx{H} \in \Cspace{KN \times KN} : \mtx{H} = \mtx{H}^\adj,
	\mtx{H}_{nn} = \Id_K \text{ for $n = 1, 2, \dots, N$}, \\
\text{and } \fnorm{ \mtx{H}_{mn} } \leq \mu \text{ for all $m \neq n$}
\}.
\end{multline}
Although the structural constraint set evidently depends on the parameter $\mu$, we will usually eliminate $\mu$ from the notation for simplicity.  We also define the spectral constraint set
\begin{equation} \label{eqn:G-chord}
\coll{G} \defby
\left\{ \mtx{G} \in \Cspace{KN \times KN} : \mtx{G} \psd \mtx{0},
	\rank{\mtx{G}} \leq d, \text{ and } \trace{\mtx{G}} = KN \right\}.
\end{equation}
Both constraint sets are closed and bounded, hence compact.  The structural constraint set $\coll{H}$ is convex, but the spectral constraint set is not.

To solve the feasibility problem \eqref{eqn:chord-feas}, we must find a matrix that lies in the intersection of $\coll{G}$ and $\coll{H}$.  This section states the algorithm, and the succeeding two sections provide some implementation details.

\begin{algorithm}[Alternating Projection] \label{alg:alt-proj}
\begin{inputs}
\item	A $KN \times KN$ Hermitian matrix $\mtx{G}^{(0)}$
\item	The maximum number of iterations $T$
\end{inputs}
\begin{outputs}
\item	A $KN \times KN$ matrix $\mtx{G}_{\mathrm{out}}$ that belongs to $\coll{G}$ and whose diagonal blocks are identity matrices
\end{outputs}
\begin{procedure}
\item	Initialize $t \leftarrow 0$.
\item	Determine a matrix $\mtx{H}^{(t)}$ that solves
$$
\min_{\mtx{H} \in \coll{H}} \smnorm{\mathrm{F}}{ \mtx{H} - \mtx{G}^{(t)} }.
$$
\item	Determine a matrix $\mtx{G}^{(t + 1)}$ that solves
$$
\min_{\mtx{G} \in \coll{G}} \smnorm{\mathrm{F}}{ \mtx{G} - \mtx{H}^{(t)} }.
$$
\item	Increment $t$.
\item	If $t < T$, return to Step 2.
\item	Define the block-diagonal matrix $\mtx{D} = \diag \mtx{G}^{(T)}$.
\item	Return the matrix
$$
\mtx{G}_{\mathrm{out}}
	= \mtx{D}^{-1/2} \mtx{G}^{(T)} \mtx{D}^{-1/2}.
$$
\end{procedure}
\end{algorithm}

The iterates generated by this algorithm are not guaranteed to converge in norm.  Therefore, we have chosen to halt the algorithm after a fixed number of steps instead of checking the behavior of the sequence of iterates.  We discuss the convergence properties of the algorithm in the sequel.

The scaling in the last step normalizes the diagonal blocks of the matrix but preserves its inertia (i.e., numbers of negative, zero, and positive eigenvalues).  Since $\mtx{G}^{(T)}$ is a positive-semidefinite matrix with rank $d$ or less, the output matrix $\mtx{G}_{\mathrm{out}}$ shares these traits.  It follows that the output matrix always admits a factorization $\mtx{G}_{\mathrm{out}} = \mtx{X}^\adj \mtx{X}$ where $\mtx{X}$ is a $d \times KN$ configuration matrix.  Property 3 is the only one of the six properties that may be violated.

\subsection{The Matrix Nearness Problems}

To implement Algorithm \ref{alg:alt-proj}, we must solve the matrix nearness problems in Steps 2 and 3.  The first one is straightforward.

\begin{prop}
Let $\mtx{G}$ be an Hermitian matrix.  With respect to the Frobenius norm, the unique matrix in $\coll{H}(\mu)$ nearest to $\mtx{G}$ has diagonal blocks equal to the identity and off-diagonal blocks that satisfy
$$
\mtx{H}_{mn} = \left\{
\begin{array}{ll}
\mtx{G}_{mn} & \qquad \text{if $\fnorm{\mtx{G}_{mn}} \leq \mu$, and} \\
\mu \, \mtx{G}_{mn} / \fnorm{\mtx{G}_{mn}} & \qquad \text{otherwise.}
\end{array}
\right.
$$
\end{prop}

It is rather more difficult to find a nearest matrix in the spectral constraint set.  To state the result, we define the plus operator by the rule $(x)_{+} = \max\{ 0, x \}$.

\begin{prop}
Let $\mtx{H}$ be an Hermitian matrix whose eigenvalue decomposition is $\sum_{j = 1}^{KN} \lambda_j \vct{u}_j \vct{u}_j^\adj$ with the eigenvalues arranged in nonincreasing order: $\lambda_1 \geq \lambda_2 \geq \dots \geq \lambda_{KN}$.  With respect to the Frobenius norm, a matrix in $\coll{G}$ closest to $\mtx{H}$ is given by
$$
\sum\nolimits_{j = 1}^d
(\lambda_j - \gamma)_{+} \vct{u}_j \vct{u}_j^\adj
$$
where the scalar $\gamma$ is chosen so that
$$
\sum\nolimits_{j = 1}^d (\lambda_j - \gamma)_{+} = KN.
$$
This best approximation is unique provided that $\lambda_d > \lambda_{d+1}$.
\end{prop}

The nearest matrix described by this theorem can be computed efficiently from an eigenvalue decomposition of $\mtx{H}$.  (See \cite{GvL96:Matrix-Computations} for computational details.)  The value of $\gamma$ is uniquely determined, but one must solve a small rootfinding problem to solve it.  The bisection method is an appropriate technique since the plus operator is nondifferentiable.  We omit the details, which are routine.

\begin{proof}
Given an Hermitian matrix $\mtx{A}$, denote by $\vct{\lambda}(\mtx{A})$ the vector of eigenvalues arranged in nonincreasing order.  Then we may decompose $\mtx{A} = \mtx{U} \{\diag \vct{\lambda}(\mtx{A}) \}\mtx{U}^\adj$ for some unitary matrix $\mtx{U}$.

Finding the matrix in $\coll{G}$ closest to $\mtx{H}$ is equivalent to solving the optimization problem
\begin{align*}
\min_{\mtx{G}} \fnormsq{ \mtx{G} - \mtx{H} }
\subjto
&\lambda_j(\mtx{G}) \geq 0 \text{ for $j = 1, \dots, d$}, \\
&\lambda_j(\mtx{G}) = 0 \text{ for $j = d + 1, \dots, KN$, and} \\
&\sum\nolimits_{j=1}^{KN} \lambda_j(\mtx{G}) = KN.
\end{align*}
First, we fix the eigenvalues of $\mtx{G}$ and minimize with respect to the unitary part of its eigenvalue decomposition.  In consequence of the Hoffman--Wielandt Theorem \cite{HJ85:Matrix-Analysis}, the objective function is bounded below:
$$
\fnormsq{ \mtx{G} - \mtx{H} } \geq
\enormsq{ \vct{\lambda}(\mtx{G}) - \vct{\lambda}(\mtx{H}) }.
$$
Equality holds if and only if $\mtx{G}$ and $\mtx{H}$ are simultaneously diagonalizable by a unitary matrix.  Therefore, if we decompose $\mtx{H} = \mtx{U} \{\diag \vct{\lambda}(\mtx{H})\} \mtx{U}^\adj$, the objective function attains its minimal value whenever $\mtx{G} = \mtx{U} \{\diag \vct{\lambda}(\mtx{G})\} \mtx{U}^\adj$.  Note that the matrix $\mtx{U}$ may not be uniquely determined.

We find the optimal vector of eigenvalues $\vct{\xi}$ for the matrix $\mtx{G}$ by solving the (strictly) convex program
\begin{align*}
\min_{\vct{\xi}} \enormsq{ \vct{\xi} - \vct{\lambda}(\mtx{H}) }
\subjto
& \xi_j \geq 0 \text{ for $j = 1, \dots, d$,} \\
& \xi_j = 0 \text{ for $j = d + 1, \dots, KN$, and } \\
& \sum\nolimits_{j=1}^{KN} \xi_j = KN.
\end{align*}
This minimization is accomplished by an application of Karush--Kuhn--Tucker theory \cite{Roc70:Convex-Analysis}.  In short, the top $d$ eigenvalues of $\mtx{H}$ are translated an equal amount, and those that become negative are set to zero.  The size of the translation is chosen to fulfill the third condition (which controls the trace of $\mtx{G}$).  The entries of the optimal $\vct{\xi}$ are nonincreasing on account of the ordering of $\vct{\lambda}(\mtx{H})$.

Finally, the uniqueness claim follows from the fact that the eigenspace associated with the top $d$ eigenvectors of $\mtx{H}$ is uniquely determined if and only if $\lambda_d(\mtx{H}) > \lambda_{d+1}(\mtx{H})$.
\end{proof}

\subsection{Choosing an Initial Configuration}

The success of the algorithm depends on adequate selection of the input matrix $\mtx{G}^{(0)}$.  We have found that the following strategy is reasonably effective.  It chooses random subspaces and adds them to the initial configuration only if they are sufficiently distant from the subspaces that have already been chosen.

\begin{algorithm}[Initial Configuration] \label{alg:starting-point}
\begin{inputs}
\item	The ambient dimension $d$, the subspace dimension $K$, and the number $N$ of subspaces
\item	An upper bound $\tau$ on the similarity between subspaces
\item	The maximum number $T$ of random selections
\end{inputs}
\begin{outputs}
\item	A $KN \times KN$ matrix $\mtx{G}$ from $\coll{G}$ whose off-diagonal blocks also satisfy $\fnorm{ \mtx{G}_{mn} } \leq \tau$
\end{outputs}
\begin{procedure}
\item	Initialize $t \leftarrow 0$ and $n \leftarrow 1$.
\item	Increment $t$.  If $t > T$, print a failure notice and stop.
\item	Pick a $d \times K$ matrix $\mtx{X}_n$ whose range is a uniformly random subspace in $\CGspace{K}{d}$.
\item	If $\fnorm{ \mtx{X}_m^\adj \mtx{X}_n } \leq \tau$ for each $m = 1, \dots, n - 1$, then increment $n$.
\item	If $n \leq N$, return to Step 2.
\item	Form the matrix $\mtx{X} =
\begin{bmatrix} \mtx{X}_1 & \mtx{X}_2 & \dots & \mtx{X}_N \end{bmatrix}$.
\item	Return the Gram matrix $\mtx{G} = \mtx{X}^\adj \mtx{X}$.
\end{procedure}
\end{algorithm}

To implement Step 3, we use the method developed in \cite{Ste80:Efficient-Generation}.  Draw a $d \times K$ matrix whose entries are iid complex, standard normal random variables, and perform a $\mtx{QR}$ decomposition.  The first $K$ columns of the unitary part of the $\mtx{QR}$ decomposition form an orthonormal basis for a random $K$-dimensional subspace.

The purpose of the parameter $\tau$ is to prevent the starting configuration $\mtx{X}$ from containing blocks that are nearly identical.  The extreme case $\tau = \sqrt{K}$ places no restriction on the similarity between blocks.  If $\tau$ is chosen too small (or if we are unlucky in our random choices), then this selection procedure may fail.  For this reason, we add an iteration counter to prevent the algorithm from entering an infinite loop.  We typically choose values of $\tau$ very close to the maximum value.

\subsection{Theoretical Behavior of Algorithm}

It is important to be aware that packing problems are typically difficult to solve.  Therefore, we cannot expect that our algorithm will necessarily produce a point in the intersection of the constraint sets.  One may ask whether we can make any guarantees about the behavior of Algorithm \ref{alg:alt-proj}.  This turns out to be difficult.  Indeed, there is potential that an alternating projection algorithm will fail to generate a convergent sequence of iterates \cite{Mey76:Sufficient-Conditions}.  Nevertheless, it can be shown that the sequence of iterates has accumulation points and that these accumulation points satisfy a weak structural property.

In practice, the alternating projection algorithm seems to converge, but a theoretical justification for this observation is lacking.  A more serious problem is that the algorithm frequently requires as many as 5000 iterations before the iterates settle down.  This is one of the major weaknesses of our approach.

For reference, we offer the best theoretical convergence result that we know.  The distance between a matrix and a compact collection of matrices is defined as
$$
\dist( \mtx{M}, \coll{C} ) \defby
\min_{\mtx{C} \in \coll{C}} \fnorm{ \mtx{M} - \mtx{C} }.
$$
It can be shown that the distance function is Lipschitz, hence continuous.

\begin{thm}[Global Convergence]
Suppose that Algorithm \ref{alg:alt-proj} generates an infinite sequence of iterates $\{ (\mtx{G}^{(t)}, \mtx{H}^{(t)}) \}$.  This sequence has at least one accumulation point.
\begin{itemize}
\item	Every accumulation point lies in $\coll{G} \times \coll{H}$.
\item	Every accumulation point $(\overline{\mtx{G}}, \overline{\mtx{H}})$ satisfies
$$
\fnorm{ \overline{\mtx{G}} - \overline{\mtx{H}} }
= \lim_{t \to \infty} \smnorm{\rm F}{ \mtx{G}^{(t)} - \mtx{H}^{(t)} }.
$$
\item	Every accumulation point $(\overline{\mtx{G}}, \overline{\mtx{H}})$ satisfies
$$
\fnorm{ \overline{\mtx{G}} - \overline{\mtx{H}} }
	= \dist(\overline{\mtx{G}}, \coll{H})
	= \dist(\overline{\mtx{H}}, \coll{G}).
$$
\end{itemize}
\end{thm}

\begin{proof}[Proof sketch]
The existence of an accumulation point follows from the compactness of the constraint sets.  The algorithm does not increase the distance between successive iterates, which is bounded below by zero.  Therefore, this distance must converge.  The distance functions are continuous, so we can take limits to obtain the remaining assertions.
\end{proof}

A more detailed treatment requires the machinery of point-to-set maps, and it would not enhance our main discussion.  Please see the appendices of \cite{TDHS05:Designing-Structured} for additional information.

\section{Bounds on the Packing diameter} \label{sec:rankin}

To assay the quality of the packings that we produce, it helps to have some upper bounds on the packing diameter.  If a configuration of subspaces has a packing diameter close to the upper bound, that configuration must be a nearly optimal packing.  This approach allows us to establish that many of the packings we construct numerically have packing diameters that are essentially optimal.

\begin{thm}[Conway--Hardin--Sloane \cite{CHS96:Packing-Lines}]
The packing diameter of $N$ subspaces in the Grassmannian manifold $\FGspace{K}{d}$ equipped with chordal distance is bounded above as
\begin{equation} \label{eqn:rankin-chord}
\pack_{\chord}(\coll{X})^2 \leq
	\frac{K (d - K)}{d} \frac{N}{N - 1}.
\end{equation}
If the bound is met, all pairs of subspaces are equidistant.  When $\mathbb{F} = \Rspace{}$, the bound is attainable only if $N \leq \half d(d+1)$.  When $\mathbb{F} = \Cspace{}$, the bound is attainable only if $N \leq d^2$.
\end{thm}

The complex case is not stated in \cite{CHS96:Packing-Lines}, but it follows from an identical argument.  We refer to \eqref{eqn:rankin-chord} as the \term{Rankin bound} for subspace packings with respect to the chordal distance.  The reason for the nomenclature is that the result is established by embedding the chordal Grassmannian manifold into a Euclidean sphere and applying the classical Rankin bound for sphere packing~\cite{Ran47:Closest-Packing}.

It is also possible to draw a corollary on packing with respect to the spectral distance; this result is novel.  A subspace packing is said to be \term{equi-isoclinic} if all the principal angles between all pairs of subspaces are identical \cite{LS73:Equi-Isoclinic-Subspaces}.

\begin{cor} \label{cor:rankin-spec}
We have the following bound on the packing diameter of $N$ subspaces in the Grassmannian manifold $\FGspace{K}{d}$ equipped with the spectral distance.
\begin{equation} \label{eqn:rankin-spec}
\pack_{\spec}(\coll{X})^2 \leq
	\frac{d - K}{d} \frac{N}{N - 1}.
\end{equation}
If the bound is met, the packing is equi-isoclinic. 
\end{cor}

We refer to \eqref{eqn:rankin-spec} as the Rankin bound for subspace packings with respect to spectral distance.

\begin{proof}
The power mean inequality (equivalently, H{\"o}lder's inequality) yields %\cite{HLP52:Inequalities}
$$
\min\nolimits_k \sin \theta_k \leq
	\left[ K^{-1} \sum\nolimits_{k=1}^K \sin^2 \theta_k \right]^{1/2}.
$$
For angles between zero and $\pi/2$, equality holds if and only if $\theta_1 = \dots = \theta_K$.  It follows that
$$
\pack_{\spec}(\coll{X})^2 \leq
K^{-1} \pack_{\chord}(\coll{X})^2 \leq
\frac{d - K}{d} \frac{N}{N - 1}.
$$
If the second inequality is met, then all pairs of subspaces are equidistant with respect to the chordal metric.  Moreover, if the first inequality is met, then the principal angles between each pair of subspaces are constant.  Together, these two conditions imply that the packing is equi-isoclinic.
\end{proof}

An upper bound on the maximum number of equi-isoclinic subspaces is available.  Its authors do not believe that it is sharp.

\begin{thm}[Lemmens--Seidel \cite{LS73:Equi-Isoclinic-Subspaces}]
\label{thm:ls-isoclinic}
The maximum number of equi-isoclinic $K$-dimensional subspaces of $\Rspace{d}$ is no greater than
$$
\half d(d+1) - \half K (K + 1) + 1.
$$
Similarly, the maximum number of equi-isoclinic $K$-dimensional subspaces of $\Cspace{d}$ does not exceed
$$
d^2 - K^2 + 1.
$$
\end{thm}

%We do not know any bounds for packings with respect to the Fubini--Study distance.

\section{Experiments} \label{sec:exp}

Our approach to packing is experimental rather than theoretical, so the real question is how Algorithm \ref{alg:alt-proj} performs in practice.  In principle, this question is difficult to resolve because the optimal packing diameter is unknown for almost all combinations of $d$ and $N$.  Whenever possible, we compared our results with the Rankin bound and with the ``world record'' packings tabulated by N.\ J.\ A.\ Sloane and his colleagues \cite{Slo:Grassmannian-Web}.  In many cases, the algorithm was able to identify a nearly optimal packing.  Moreover, it yields interesting results for packing problems that have not received numerical attention.

In the next subsection, we describe detailed experiments on packing in real and complex projective spaces.  Then, we move on to packings of subspaces with respect to the chordal distance.  Afterward, we study the spectral distance and the Fubini--Study distance.

\subsection{Projective Packings}

Line packings are the simplest type of Grassmannian packing, so they offer a natural starting point.  Our goal is to produce the best packing of $N$ lines in $\FPspace{d - 1}$.  In the real case, Sloane's tables allow us to determine how much our packings fall short of the world record.  In the complex setting, there is no comparable resource, so we must rely on the Rankin bound to gauge how well the algorithm performs.

Let us begin with packing in real projective spaces.  We attempted to construct configurations of real lines whose maximum absolute inner product $\mu$ fell within $10^{-5}$ of the best value tabulated in \cite{Slo:Grassmannian-Web}.  For pairs $(d, N)$ with $d = 3, 4, 5$ and $N = 4, 5, \dots, 25$, we computed the putatively optimal value of the feasibility parameter $\mu$ from Sloane's data and equation \eqref{eqn:feas-parm}.  In each of 10 trials, we constructed a starting matrix using Algorithm \ref{alg:starting-point} with parameters $\tau = 0.9$ and $T = 10,000$.  (Recall that the value of $T$ determines the maximum number of random subspaces that are drawn when trying to construct the initial configuration.)  We applied alternating projection, Algorithm \ref{alg:alt-proj}, with the computed value of $\mu$ and the maximum number of iterations $T = 5000$.  (Our numerical experience indicates that increasing the maximum number of iterations beyond 5000 does not confer a significant benefit.)  We halted the iteration in Step 4 if the iterate $\mtx{G}^{(t)}$ exhibited no off-diagonal entry with absolute value greater than $\mu + 10^{-5}$.   After 10 trials, we recorded the largest packing diameter attained, as well as the average value of the packing diameter.  We also recorded the average number of iterations the alternating projection required per trial.

Table \ref{tab:real-projective} delivers the results of this experiment.  Following Sloane, we have reported the degrees of arc subtended by the closest pair of lines.  We believe that it is easiest to interpret the results geometrically when they are stated in this fashion.  All the tables and figures related to packing are collated at the back of this paper for easy comparison.

According to the table, the best configurations produced by alternating projection consistently attain packing diameters tenths or hundredths of a degree away from the best configurations known.  The average configurations returned by alternating projection are slightly worse, but they usually fall within a degree of the putative optimal.  Moreover, the algorithm finds certain configurations with ease.  For the pair $(5, 16)$, fewer than 1000 iterations are required on average to achieve a packing within 0.001 degrees of optimal.

A second observation is that the alternating projection algorithm typically performs better when the number $N$ of points is small.  The largest errors are all clustered at larger values of $N$.  A corollary observation is that the average number of iterations per trial tends to increase with the number of points.

%We believe that the explanation for these phenomena is that packing problems have a combinatorial regime, where solutions have a lot of symmetry and structure, and a random regime, where the solutions have very little order.  The algorithm typically seems to perform better in the combinatorial regime, although it fails for certain unusually structured ensembles.

There are several anomalies that we would like to point out.  The most interesting pathology occurs at the pair $(d, N) = (5, 19)$.  The best packing diameter calculated by alternating projection is about $1.76^\circ$ worse than the optimal configuration, and it is also $1.76^\circ$ worse than the best packing diameter computed for the pair $(5, 20)$.  From Sloane's tables, we can see that the (putative) optimal packing of 19 lines in $\RPspace{4}$ is actually a subset of the best packing of 20 lines.  Perhaps the fact that this packing is degenerate makes it difficult to construct.  A similar event occurs (less dramatically) at the pair $(5, 13)$.  The table also shows that the algorithm performs less effectively when the number of lines exceeds 20.

In complex projective spaces, this methodology does not apply because there are no tables available.  In fact, we only know of one paper that contains numerical work on packing in complex projective spaces \cite{ARU01:Multiple-Antenna-Signal}, but it gives very few examples of good packings.  The only method we know for gauging the quality of a complex line packing is to compare it against an upper bound.  The Rankin bound for projective packings, which is derived in Section \ref{sec:rankin}, states that every configuration $\coll{X}$ of $N$ lines in either $\RPspace{d-1}$ or $\CPspace{d-1}$ satisfies the inequality
$$
\pack_{\mathbb{P}}(\coll{X})^2 \leq
\frac{ (d- 1) \, N }{ d \, (N - 1) }.
$$
This bound is attainable only for rare combinations of $d$ and $N$.  In particular, the bound can be met in $\RPspace{d-1}$ only if $N \leq \half \, d \, (d+1)$.  In the space $\CPspace{d-1}$, attainment requires that $N \leq d^2$.  Any arrangement of lines that meets the Rankin bound must be equiangular.  These optimal configurations are called \term{equiangular tight frames}.  See \cite{SH03:Grassmannian-Frames,HP04:Optimal-Frames,TDHS05:Designing-Structured,STDH07:Existence-Equiangular} for more details.

We performed some \lang{ad hoc} experiments to produce configurations of complex lines with large packing diameters.  For each pair $(d, N)$, we used the Rankin bound to determine a lower limit on the feasibility parameter $\mu$.  Starting matrices were constructed with Algorithm \ref{alg:starting-point} using values of $\tau$ ranging between 0.9 and 1.0.  (Algorithm \ref{alg:starting-point} typically fails for smaller values of $\tau$.)  For values of the feasibility parameter between the minimal value and twice the minimal value, we performed 5000 iterations of Algorithm \ref{alg:alt-proj}, and we recorded the largest packing diameter attained during these trials.

Table \ref{tab:complex-projective} compares our results against the Rankin bound.  We see that many of the complex line configurations have packing diameters much smaller than the Rankin bound, which is not surprising because the bound is usually not attainable.  Some of our configurations fall within a thousandth of a degree of the bound, which is essentially optimal.

Table \ref{tab:complex-projective} contains a few oddities.  In $\CPspace{4}$, the best packing diameter computed for $N = 18, 19, \dots, 24$ is worse than the packing diameter for $N = 25$.  This configuration of 25 lines is an equiangular tight frame, which means that it is an optimal packing \cite[Table 1]{TDHS05:Designing-Structured}.  It seems likely that the optimal configurations for the preceding values of $N$ are just subsets of the optimal arrangement of 25 lines.  As before, it may be difficult to calculate this type of degenerate packing.  A similar event occurs less dramatically at the pair $(d, N) = (4,13)$ and at the pairs $(4,17)$ and $(4,18)$.

Figure \ref{fig:proj-packings} compares the quality of the best real projective packings from \cite{Slo:Grassmannian-Web} with the best complex projective packings that we obtained.  It is natural that the complex packings are better than the real packings because the real projective space can be embedded isometrically into the complex projective space.  But it is remarkable how badly the real packings compare with the complex packings.  The only cases where the real and complex ensembles have the same packing diameter occur when the real configuration meets the Rankin bound.

\subsection{The Chordal Distance}

Emboldened by this success with projective packings, we move on to packings of subspaces with respect to the chordal distance.  Once again, we are able to use Sloane's tables for guidance in the real case.  In the complex case, we fall back on the Rankin bound.

For each triple $(d, K, N)$, we determined a value for the feasibility parameter $\mu$ from the best packing diameter Sloane recorded for $N$ subspaces in $\RGspace{K}{d}$, along with equation \eqref{eqn:feas-parm}.  We constructed starting points using the modified version of Algorithm \ref{alg:starting-point} with $\tau = \sqrt{K}$, which represents no constraint.  (We found that the alternating projection performed no better with initial configurations generated from smaller values of $\tau$.)  Then we executed Algorithm \ref{alg:alt-proj} with the calculated value of $\mu$ for 5000 iterations.

Table \ref{tab:real-chordal} demonstrates how the best packings we obtained compare with Sloane's best packings.  Many of our real configurations attained a squared packing diameter within $10^{-3}$ of the best value Sloane recorded.  Our algorithm was especially successful for smaller numbers of subspaces, but its performance began to flag as the number of subspaces approached 20.

Table \ref{tab:real-chordal} contains several anomalies.  For example, our configurations of $N = 11, 12, \dots, 16$ subspaces in $\Rspace{4}$ yield worse packing diameters than the configuration of 17 subspaces.  It turns out that this configuration of 17 subspaces is optimal, and Sloane's data show that the (putative) optimal arrangements of 11 to 16 subspaces are all subsets of this configuration.  This is the same problem that occurred in some of our earlier experiments, and it suggests again that our algorithm has difficulty locating these degenerate configurations precisely.

The literature contains very few experimental results on packing in complex Grassmannian manifolds equipped with chordal distance.  To our knowledge, the only numerical work appears in two short tables from \cite{ARU01:Multiple-Antenna-Signal}.  Therefore, we found it valuable to compare our results against the Rankin bound for subspace packings, which is derived in Section \ref{sec:rankin}.  For reference, this bound requires that every configuration $\coll{X}$ of $N$ subspaces in $\FGspace{K}{d}$ satisfy the inequality
$$
\pack_{\chord}(\coll{X})^2 \leq
\frac{K \, (d - K)}{d} \, \frac{N}{N - 1}.
$$
This bound cannot always be met.  In particular, the bound is attainable in the complex setting only if $N \leq d^2$.  In the real setting, the bound requires that $N \leq \half \, d \, (d + 1)$.  When the bound is attained, each pair of subspaces in $\coll{X}$ is equidistant.

We performed some \lang{ad hoc} experiments to construct a table of packings in $\CGspace{K}{d}$ equipped with the chordal distance.  For each triple $(d, K, N)$, we constructed random starting points using Algorithm \ref{alg:starting-point} with $\tau = \sqrt{K}$ (which represents no constraint).  Then we used the Rankin bound to calculate a lower limit on the feasibility parameter $\mu$.  For this value of $\mu$, we executed the alternating projection, Algorithm \ref{alg:alt-proj}, for 5000 iterations.

The best packing diameters we obtained are listed in Table \ref{tab:complex-chordal}.  We see that there is a remarkable correspondence between the squared packing diameters of our configurations and the Rankin bound.  Indeed, many of our packings are within $10^{-4}$ of the bound, which means that these configurations are essentially optimal.  The algorithm was less successful as $N$ approached $d^2$, which is an upper bound on the number $N$ of subspaces for which the Rankin bound is attainable.

Figure \ref{fig:chord-packings} compares the packing diameters of the best configurations in real and complex Grassmannian spaces equipped with chordal distance.  It is remarkable that both real and complex packings almost meet the Rankin bound for all $N$ where it is attainable.  Notice how the real packing diameters fall off as soon as $N$ exceeds $\half \, d \, (d + 1)$.  In theory, a complex configuration should always attain a better packing diameter than the corresponding real configuration because the real Grassmannian space can be embedded isometrically into the complex Grassmannian space.  The figure shows that our best arrangements of 17 and 18 subspaces in $\CGspace{2}{4}$ are actually slightly worse than the real arrangements calculated by Sloane.  This indicates a failure of the alternating projection algorithm.

\subsection{The Spectral Distance}

Next, we consider how to compute Grassmannian packings with respect to the spectral distance.  This investigation requires some small modifications to the algorithm, which are described in the next subsection.  Afterward, we provide the results of some numerical experiments.

\subsubsection{Modifications to Algorithm}

To construct packings with respect to the spectral distance, we tread a familiar path.  Suppose that we wish to produce a configuration of $N$ subspaces in $\CGspace{K}{d}$ with a packing diameter $\rho$.  The feasibility problem requires that
\begin{equation} \label{eqn:grass-proj-feasibility}
\max_{m \neq n} \ \pnorm{2, 2}{ \mtx{X}_m^\adj \, \mtx{X}_n }
\leq \mu
\end{equation}
where $\mu = \sqrt{ 1 - \rho^2 }$.  This leads to the convex structural constraint set
\begin{multline*} %\label{eqn:chord-structural-set}
\coll{H}(\mu) \defby
\{ \mtx{H} \in \Cspace{KN \times KN} : \mtx{H} = \mtx{H}^\adj, \quad
\mtx{H}_{nn} = \Id \text{ for $n = 1, 2, \dots, N$},
\quad\text{and}\quad \\
\pnorm{2,2}{\mtx{H}_{mn}} \leq \mu \text{ for all $m \neq n$}
\}.
\end{multline*}
The spectral constraint set is the same as before.  The next proposition shows how to find the matrix in $\coll{H}$ closest to an initial matrix.  In preparation, define the truncation operator $[x]_{\mu} = \min\{ x, \mu \}$ for numbers, and extend it to matrices by applying it to each component.

\begin{prop}
Let $\mtx{G}$ be an Hermitian matrix.  With respect to the Frobenius norm, the unique matrix in $\coll{H}(\mu)$ nearest to $\mtx{G}$ has a block identity diagonal.  If the off-diagonal block $\mtx{G}_{mn}$ has a singular value decomposition $\mtx{U}_{mn} \mtx{C}_{mn} \mtx{V}_{mn}^\adj$, then
$$
\mtx{H}_{mn} = \left\{
\begin{array}{ll}
\mtx{G}_{mn} & \text{if $\pnorm{2,2}{\mtx{G}_{mn}} \leq \mu$, and} \\
\mtx{U}_{mn} \, [\mtx{C}_{mn}]_{\mu} \, \mtx{V}_{mn}^\adj \qquad &
\text{otherwise}.
\end{array} \right.
$$
\end{prop}

\begin{proof}
To determine the $(m, n)$ off-diagonal block of the solution matrix $\mtx{H}$, we must solve the optimization problem
$$
\min\nolimits_{\mtx{A}} \ \half \, \fnormsq{ \mtx{A} - \mtx{G}_{mn} }
\subjto
\pnorm{2,2}{\mtx{A}} \leq \mu.
$$
The Frobenius norm is strictly convex and the spectral norm is convex, so this problem has a unique solution.

Let $\vct{\sigma}(\cdot)$ return the vector of decreasingly ordered singular values of a matrix.  Suppose that $\mtx{G}_{mn}$ has the singular value decomposition $\mtx{G}_{mn} = \mtx{U} \{ \diag \vct{\sigma}(\mtx{G}_{mn}) \} \mtx{V}^\adj$.  The constraint in the optimization problem depends only on the singular values of $\mtx{A}$, and so the Hoffman--Wielandt Theorem for singular values \cite{HJ85:Matrix-Analysis} allows us to check that the solution has the form $\mtx{A} = \mtx{U} \{ \diag \vct{\sigma}(\mtx{A}) \} \mtx{V}^\adj$.

To determine the singular values $\vct{\xi} = \vct{\sigma}(\mtx{A})$ of the solution, we must solve the (strictly) convex program
$$
\min\nolimits_{\vct{\xi}} \ \half \, \enormsq{ \vct{\xi} - \vct{\sigma}(\mtx{G}_{mn}) }
\subjto
\xi_k \leq \mu.
$$
An easy application of Karush--Kuhn--Tucker theory \cite{Roc70:Convex-Analysis} proves that the solution is obtained by truncating the singular values of $\mtx{G}_{mn}$ that exceed $\mu$.
\end{proof}

\subsubsection{Numerical Results}

To our knowledge, there are no numerical studies of packing in Grassmannian spaces equipped with spectral distance.  To gauge the quality of our results, we compare them against the upper bound of Corollary \ref{cor:rankin-spec}.  In the real or complex setting, a configuration $\coll{X}$ of $N$ subspaces in $\FGspace{K}{d}$ with respect to the spectral distance must satisfy the bound
$$
\pack_{\spec}(\coll{X})^2 \leq
\frac{d - K}{d} \, \frac{N}{N - 1}.
$$
In the real case, the bound is attainable only if $N \leq \half \, d \, (d + 1) - \half \, K \, (K + 1) + 1$, while attainment in the complex case requires that $N \leq d^2 - K^2 + 1$ \cite{LS73:Equi-Isoclinic-Subspaces}.  When a configuration meets the bound, the subspaces are not only equidistant but also \term{equi-isoclinic}.  That is, all principal angles between all pairs of subspaces are identical.

We performed some limited \lang{ad hoc} experiments in an effort to produce good configurations of subspaces with respect to the spectral distance.  We constructed random starting points using the modified version of Algorithm \ref{alg:starting-point} with $\tau = 1$, which represents no constraint.  (Again, we did not find that smaller values of $\tau$ improved the performance of the alternating projection.) From the Rankin bound, we calculated the smallest possible value of the feasibility parameter $\mu$.  For values of $\mu$ ranging from the minimal value to twice the minimal value, we ran the alternating projection, Algorithm \ref{alg:alt-proj}, for 5000 iterations, and we recorded the best packing diameters that we obtained.

Table \ref{tab:projection-dist} displays the results of our calculations.  We see that some of our configurations essentially meet the Rankin Bound, which means that they are equi-isoclinic.  It is clear that alternating projection also succeeds reasonably well for this packing problem.

The most notable pathology in the table occurs for configurations of 8 and 9 subspaces in $\RGspace{3}{6}$.  In these cases, the algorithm always yielded arrangements of subspaces with a zero packing diameter, which implies that two of the subspaces intersect nontrivially.  Nevertheless, we were able to construct random starting points with a nonzero packing diameter, which means that the algorithm is making the initial configuration worse.  We do not understand the reason for this failure.

Figure \ref{fig:spec-packings} makes a graphical comparison between the real and complex subspace packings.  On the whole, the complex packings are much better than the real packings.  For example, every configuration of subspaces in $\CGspace{2}{6}$ nearly meets the Rankin bound, while just two of the real configurations achieve the same distinction.  In comparison, it is curious how few arrangements in $\CGspace{2}{5}$ come anywhere near the Rankin bound.

%Since the Rankin Bound is never attained for values of The graphs also suggest the potential for tightening the upper bound on the number of subspaces at which the Rankin Bound is attainable.  (In their study \cite{LS72:Equi-Isoclinic-Subspaces} of equi-isoclinic subspaces, Lemmens and Seidel reach the same conclusion from a theoretical perspective.)

%We have performed some limited experiments on packing subspaces with respect to the spectral distance, and we have compared our results against the upper bound \eqref{eqn:rankin-proj}.  Turn to Table \ref{tab:projection-dist} for the calculations.  Meanwhile, Figure \ref{fig:spec-packings} illustrates the differences between the real packings and the complex packings.  Many of our configurations very nearly attained the upper bound, which implies that the subspaces are not only equidistant but also equi-isoclinic.  That is, all principal angles between all pairs of subspaces are identical.  We should mention that the algorithm failed to produce reasonable packings of 8 and 9 subspaces in $\RGspace{3}{6}$.

\subsection{The Fubini--Study Distance}

When we approach the problem of packing in Grassmannian manifolds equipped with the Fubini--Study distance, we are truly out in the wilderness.  To our knowledge, the literature contains neither experimental nor theoretical treatments of this question.  Moreover, we are not presently aware of general upper bounds on the Fubini--Study packing diameter that we might use to assay the quality of a configuration of subspaces.  Nevertheless, we attempted a few basic experiments.  The investigation entails some more modifications to the algorithm, which are described below.  Afterward, we go over our experimental results.  We view this work as very preliminary.

\subsubsection{Modifications to Algorithm}

Suppose that we wish to construct a configuration of $N$ subspaces whose Fubini--Study packing diameter exceeds $\rho$.  The feasibility condition is
\begin{equation} \label{eqn:grass-fs-feasibility}
\max_{m \neq n} \ \abs{\det \mtx{X}_m^\adj \, \mtx{X}_n }
\leq \mu
\end{equation}
where $\mu = \cos \rho$.  This leads to the structural constraint set
\begin{multline*} %\label{eqn:chord-structural-set}
\coll{H}(\mu) \defby
\{ \mtx{H} \in \Cspace{KN \times KN} : \mtx{H} = \mtx{H}^\adj, \quad
\mtx{H}_{nn} = \Id \text{ for $n = 1, 2, \dots, N$},
\quad\text{and}\quad \\
\abs{\det{\mtx{H}_{mn}}} \leq \mu \text{ for all $m \neq n$}
\}.
\end{multline*}
Unhappily, this set is no longer convex.  To produce a nearest matrix in $\coll{H}$, we must solve a nonlinear programming problem.  The following proposition describes a numerically favorable formulation.

\begin{prop}
Let $\mtx{G}$ be an Hermitian matrix.  Suppose that the off-diagonal block $\mtx{G}_{mn}$ has singular value decomposition $\mtx{U}_{mn} \mtx{C}_{mn} \mtx{V}_{mn}^\adj$.  Let $\vct{c}_{mn} = \diag \mtx{C}_{mn}$, and find a (real) vector $\vct{x}_{mn}$ that solves the optimization problem
$$
\min_{\vct{x}} \quad \half \enormsq{ \exp(\vct{x}) - \vct{c}_{mn} }
\subjto
\onevct^\adj \, \vct{x} \leq \log \mu.
$$
In Frobenius norm, a matrix $\mtx{H}$ from $\coll{H}(\mu)$ that is closest to $\mtx{G}$ has a block-identity diagonal and off-diagonal blocks
$$
\mtx{H}_{mn} = \left\{
\begin{array}{ll}
\mtx{G}_{mn} & \text{if $\abs{\det{\mtx{G}_{mn}}} \leq \mu$, and} \\
\mtx{U}_{mn} \{ \diag (\exp\vct{x}_{mn}) \} \mtx{V}_{mn}^\adj \qquad &
\text{otherwise}.
\end{array} \right.
$$
\end{prop}

We use $\exp(\cdot)$ to denote the componentwise exponential of a vector.  One may establish that the optimization problem is not convex by calculating the Hessian of the objective function.

\begin{proof}
To determine the $(m, n)$ off-diagonal block of the solution matrix $\mtx{H}$, we must solve the optimization problem
$$
\min\nolimits_{\mtx{A}} \ \half \, \fnormsq{ \mtx{A} - \mtx{G}_{mn} }
\subjto
\abs{\det \mtx{A} } \leq \mu.
$$
We may reformulate this problem as
$$
\min\nolimits_{\mtx{A}} \ 
\half \, \fnormsq{ \mtx{A} - \mtx{G}_{mn} }
\subjto
\sum\nolimits_{k=1}^{K} \log \sigma_k(\mtx{A}) \leq \log \mu.
$$
A familiar argument proves that the solution matrix has the same left and right singular vectors as $\mtx{G}_{mn}$.  To obtain the singular values $\vct{\xi} = \vct{\sigma}(\mtx{A})$ of the solution, we consider the mathematical program
$$
\min\nolimits_{\vct{\xi}} \ 
\half \,\enormsq{ \vct{\xi} - \vct{\sigma}(\mtx{G}_{mn}) }
\subjto
\sum\nolimits_{k=1}^K \log \xi_k \leq \log \mu.
$$
Change variables to complete the proof.
\end{proof}

\subsubsection{Numerical Experiments}

We implemented the modified version of Algorithm \ref{alg:alt-proj} in Matlab, using the built-in nonlinear programming software to solve the optimization problem required by the proposition.  For a few triples $(d, K, N)$, we ran 100 to 500 iterations of the algorithm for various values of the feasibility parameter $\mu$.  (Given the exploratory nature of these experiments, we found that the implementation was too slow to increase the number of iterations.)

The results appear in Table \ref{tab:fubini-study}.  For small values of $N$, we find that the packings exhibit the maximum possible packing diameter $\pi / 2$, which shows that the algorithm is succeeding in these cases.  For larger values of $N$, we are unable to judge how close the packings might decline from optimal.

Figure \ref{fig:fs-packings} compares the quality of our real packings against our complex packings.  In each case, the complex packing is at least as good as the real packing, as we would expect.  The smooth decline in the quality of the complex packings suggests that there is some underlying order to the packing diameters, but it remains to be discovered.

To perform large-scale experiments, it will probably be necessary to tailor an algorithm that can solve the nonlinear programming problems more quickly.  It may also be essential to implement the alternating projection in a programming environment more efficient than Matlab.  Therefore, a detailed study of packing with respect to the Fubini--Study distance must remain a topic for future research.

%\section{Theoretical Connections}

%\subsection{Real ETFs}

%\subsection{Complex ETFs}

\section{Discussion} \label{sec:open}

\subsection{Subspace Packing in Wireless Communications} \label{sec:discussion}

Configurations of subspaces arise in several aspects of wireless communication, especially in systems with multiple transmit and receive antennas. The intuition behind this connection is that the transmitted and received signals in a multiple antenna system are connected by a matrix transformation, or \term{matrix channel}.

Subspace packings occur in two wireless applications: noncoherent communication and in subspace quantization. The noncoherent application is primarily of theoretical interest, while subspace quantization has a strong impact on practical wireless systems. Grassmannian packings appear in these situations due to an assumption that the matrix channel should be modeled as a complex Gaussian random matrix.

In the noncoherent communication problem, it has been shown that, from an information-theoretic perspective, under certain assumptions about the channel matrix, the optimum transmit signal corresponds to a packing in $\CGspace{K}{d}$ where $K$ corresponds to the minimum of the number of transmit and receive antennas and $d$ corresponds to the number of consecutive samples over which the channel is constant~\cite{ZheTse:Communication-on-the-Grassmann-manifold::02,HocMar:Unitary-space-time-modulation:00}.  In other words, the number of subspaces $K$ is determined by the system configuration, while $d$ is determined by the carrier frequency and the degree of mobility in the propagation channel.

On account of this application, several papers have investigated the problem of finding packings in Grassmannian manifolds.  One approach for the case of $K = 1$ is presented in \cite{HocMar:Unitary-space-time-modulation:00}.  This paper proposes a numerical algorithm for finding line packings, but it does not discuss its properties or connect it with the general subspace packing problem.  Another approach, based on discrete Fourier transform matrices, appears in \cite{HocMarRic:Systematic-design-of-unitary:00}. This construction is both structured and flexible, but it does not lead to optimal packings.  The paper \cite{ARU01:Multiple-Antenna-Signal} studies Grassmannian packings in detail, and it contains an algorithm for finding packings in the complex Grassmannian manifold equipped with chordal distance.  This algorithm is quite complex: it uses surrogate functionals to solve a sequence of relaxed nonlinear programs.  The authors tabulate several excellent chordal packings, but it is not clear whether their method generalizes to other metrics.

The subspace quantization problem also leads to Grassmannian packings.  In multiple-antenna wireless systems, one must quantize the dominant subspace in the matrix communication channel.  Optimal quantizers can be viewed as packings in $\CGspace{K}{d}$, where $K$ is the dimension of the subspace and $d$ is the number of transmit antennas.  The chordal distance, the spectral distance, and the Fubini--Study distance are all useful in this connection \cite{LovHea:Limited-feedback-unitary-stbc:05,LovHea:Limited-feedback-unitary:05}.  This literature does not describe any new algorithms for constructing packings; it leverages results from the noncoherent communication literature.  Communication strategies based on quantization via subspace packings have been incorporated into at least one recent standard~\cite{standard16e}.

\subsection{Conclusions}

We have shown that the alternating projection algorithm can be used to solve many different packing problems.  The method is easy to understand and to implement, even while it is versatile and powerful.  In cases where experiments have been performed, we have often been able to match the best packings known.  Moreover, we have extended the method to solve problems that have not been studied numerically.  Using the Rankin bounds, we have been able to show that many of our packings are essentially optimal.  It seems clear that alternating projection is an effective numerical algorithm for packing.
%  We hope that it will allow researchers to explore \lang{terra incognita}.

\appendix

\section{Tammes' Problem} \label{app:tammes}

The alternating projection method can also be used to study Tammes' Problem of packing points on a sphere \cite{Tam30:Origin-Number}.  This question has received an enormous amount of attention over the last 75 years, and extensive tables of putatively optimal packings are available \cite{Slo:Sphere-Packing-Web}.  This appendix offers a brief treatment of our work on this problem.

\subsection{Modifications to Algorithm}

Suppose that we wish to produce a configuration of $N$ points on the unit sphere $\Sspace{d-1}$ with a packing diameter $\rho$.  The feasibility problem requires that
\begin{equation} \label{eqn:tammes-feasibility}
\max_{m \neq n} \ \ip{ \vct{x}_m }{ \vct{x}_n }
\leq \mu
\end{equation}
where $\mu = \sqrt{ 1 - \rho^2 }$.  This leads to the convex structural constraint set
\begin{multline*} %\label{eqn:chord-structural-set}
\coll{H}(\mu) \defby
\{ \mtx{H} \in \Rspace{N \times N} : \mtx{H} = \mtx{H}^\adj, \quad
h_{nn} = 1 \text{ for $n = 1, 2, \dots, N$},
\quad\text{and}\quad \\
-1 \leq h_{mn} \leq \mu \text{ for all $m \neq n$}
\}.
\end{multline*}
The spectral constraint set is the same as before.  The associated matrix nearness problem is trivial to solve.

\begin{prop} \label{prop:sphere-structural}
Let $\mtx{G}$ be a real, symmetric matrix.  With respect to Frobenius norm, the unique matrix in $\coll{H}(\mu)$ closest to $\mtx{G}$ has a unit diagonal and off-diagonal entries that satisfy
$$
h_{mn} = \left\{ \begin{array}{ll}
	-1, \qquad & g_{mn} < -1, \\
	g_{mn}, \qquad & -1 \leq g_{mn} \leq \mu, \text{ and} \\
	\mu, & \mu < g_{mn}.
	\end{array} \right.
$$
\end{prop}

\subsection{Numerical Results}

Tammes' Problem has been studied for 75 years, and many putatively optimal configurations are available.  Therefore, we attempted to produce packings whose maximum inner product $\mu$ fell within $10^{-5}$ of the best value tabulated by N.\ J.\ A.\ Sloane and his colleagues \cite{Slo:Sphere-Packing-Web}.  This resource draws from all the experimental and theoretical work on Tammes' Problem, and it should be considered the gold standard.

Our experimental setup echoes the setup for real projective packings.  We implemented the algorithms in Matlab, and we performed the following experiment for pairs $(d, N)$ with $d = 3,4,5$ and $N = 4, 5, \dots, 25$.  First, we computed the putatively optimal maximum inner product $\mu$ using the data from \cite{Slo:Sphere-Packing-Web}.  In each of 10 trials, we constructed a starting matrix using Algorithm \ref{alg:starting-point} with parameters $\tau = 0.9$ and $T = 10,000$.  Then, we executed the alternating projection, Algorithm \ref{alg:alt-proj}, with the calculated value of $\mu$ and the maximum number of iterations set to $T = 5000$.  We stopped the alternating projection in Step 4 if the iterate $\mtx{G}^{(t)}$ contained no off-diagonal entry greater than $\mu + 10^{-5}$ and proceeded with Step 6.  After 10 trials, we recorded the largest packing diameter attained, as well as the average value of the packing diameter.  We also recorded the average number of iterations the alternating projection required during each trial.

Table \ref{tab:sphere-packing} provides the results of this experiment.  The most striking feature of Table \ref{tab:sphere-packing} is that the best configurations returned by alternating projection consistently attain packing diameters that fall hundredths or thousandths of a degree away from the best packing diameters recorded by Sloane.  If we examine the maximum inner product in the configuration instead, the difference is usually on the order of $10^{-4}$ or $10^{-5}$, which we expect based on our stopping criterion.  The average-case results are somewhat worse.  Nevertheless, the average configuration returned by alternating projection typically attains a packing diameter only several tenths of a degree away from optimal.

A second observation is that the alternating projection algorithm typically performs better when the number of points $N$ is small.  The largest errors are all clustered at larger values of $N$.  A corollary observation is that the average number of iterations per trial tends to increase with the number of points.  We believe that the explanation for these phenomena is that Tammes' Problem has a combinatorial regime, where solutions have a lot of symmetry and structure, and a random regime, where the solutions have very little order.  The algorithm typically seems to perform better in the combinatorial regime, although it fails for certain unusually structured ensembles.

This claim is supported somewhat by theoretical results for $d = 3$.  Optimal configurations have only been established for $N = 1, 2, \dots, 12$ and $N = 24$.  Of these, the cases $N = 1, 2, 3$ are trivial.  The cases $N = 4, 6, 8, 12, 24$ fall from the vertices of various well-known polyhedra.  The cases $N = 5, 11$ are degenerate, obtained by leaving a point out of the solutions for $N = 6, 12$.  The remaining cases involve complicated constructions based on graphs \cite{EZ01:Spherical-Codes}.  The algorithm was able to calculate the known optimal configurations to a high order of accuracy, but it generally performed slightly better for the nondegenerate cases.

On the other hand, there is at least one case where the algorithm failed to match the optimal packing diameter, even though the optimal configuration is highly symmetric.  The best arrangement of 24 points on $\Sspace{3}$ locates them at vertices of a very special polytope called the 24-cell \cite{Slo:Sphere-Packing-Web}.  The best configuration produced by the algorithm has a packing diameter $1.79^\circ$ worse.  It seems that this optimal configuration is very difficult for the algorithm to find.  Less dramatic failures occurred at pairs $(d, N) = (3, 25)$, $(4, 14)$, $(4, 25)$, $(5, 22)$, and $(5, 23)$.  But in each of these cases, our best packing declined more than a tenth of a degree from the best recorded.

\newpage

\section{Tables and Figures} \label{app:tables}

Our experiments resulted in tables of packing diameters.  We did not store the configurations produced by the algorithm.  The Matlab code that produced these data is available on request from \texttt{jtropp@acm.caltech.edu}.

These tables and figures are intended only to describe the results of our experiments; it is likely that many of the packing diameters could be improved with additional effort.  In all cases, we present the results of calculations for the stated problem, even if we obtained a better packing by solving a different problem.  For example, a complex packing should always improve on the corresponding real packing.  If the numbers indicate otherwise, it just means that the complex experiment yielded an inferior result.  As a second example, the optimal packing diameter must not decrease as the number of points increases.  When the numbers indicate otherwise, it means that running the algorithm with more points yielded a better result than running it with fewer.  These failures may reflect the difficulty of various packing problems.

\listoftables
\listoffigures

% Table of real projective packings

\newpage

\begin{center}
\tablefirsthead{%
	\cline{3-8}
	\multicolumn{2}{c|}{} &
	\multicolumn{5}{|c|}{\textsc{Packing diameters (Degrees)}} &
	\multicolumn{1}{|c|}{\textsc{Iterations}} \\
	\hline
	$d$ & $N$ & NJAS & Best of 10 & \phantom{12} Error &
	Avg.\ of 10 & \phantom{12} Error & Avg.\ of 10 \\
	\hline\hline}
\tablehead{%
	\multicolumn{8}{l}{\textsl{\dots continued}} \\
	\multicolumn{8}{l}{} \\
	\cline{3-8}
	\multicolumn{2}{c|}{} &
	\multicolumn{5}{|c|}{\textsc{Packing diameters (Degrees)}} &
	\multicolumn{1}{|c|}{\textsc{Iterations}} \\
	\hline
	$d$ & $N$ & NJAS & Best of 10 & \phantom{12} Error &
	Avg.\ of 10 & \phantom{12} Error & Avg.\ of 10 \\
	\hline\hline}
\tabletail{%
	\hline
	\multicolumn{8}{r}{} \\
	\multicolumn{8}{r}{\textsl{continued\dots}} \\}
\tablelasttail{\hline}

\topcaption[Packing in real projective spaces]{\textsc{Packing in real projective spaces:}  For collections of $N$ points in the real projective space $\RPspace{d - 1}$, this table lists the best packing diameter (in degrees) and the average packing diameter (in degrees) obtained during ten random trials of the alternating projection algorithm.  The error columns record how far our results decline from the putative optimal packings (NJAS) reported in \cite{Slo:Grassmannian-Web}.  The last column gives the average number of iterations of alternating projection per trial before the termination condition is met.} \label{tab:real-projective}

\small
\begin{supertabular}{|cr|r|rr|rr|r|}
3	&	4	&	70.529	&	70.528	&	0.001	&	70.528	&	0.001	&	54	\\
3	&	5	&	63.435	&	63.434	&	0.001	&	63.434	&	0.001	&	171	\\
3	&	6	&	63.435	&	63.435	&	0.000	&	59.834	&	3.601	&	545	\\
3	&	7	&	54.736	&	54.735	&	0.001	&	54.735	&	0.001	&	341	\\
3	&	8	&	49.640	&	49.639	&	0.001	&	49.094	&	0.546	&	4333	\\
3	&	9	&	47.982	&	47.981	&	0.001	&	47.981	&	0.001	&	2265	\\
3	&	10	&	46.675	&	46.674	&	0.001	&	46.674	&	0.001	&	2657	\\
3	&	11	&	44.403	&	44.402	&	0.001	&	44.402	&	0.001	&	2173	\\
3	&	12	&	41.882	&	41.881	&	0.001	&	41.425	&	0.457	&	2941	\\
3	&	13	&	39.813	&	39.812	&	0.001	&	39.522	&	0.291	&	4870	\\
3	&	14	&	38.682	&	38.462	&	0.221	&	38.378	&	0.305	&	5000	\\
3	&	15	&	38.135	&	37.934	&	0.201	&	37.881	&	0.254	&	5000	\\
3	&	16	&	37.377	&	37.211	&	0.166	&	37.073	&	0.304	&	5000	\\
3	&	17	&	35.235	&	35.078	&	0.157	&	34.821	&	0.414	&	5000	\\
3	&	18	&	34.409	&	34.403	&	0.005	&	34.200	&	0.209	&	5000	\\
3	&	19	&	33.211	&	33.107	&	0.104	&	32.909	&	0.303	&	5000	\\
3	&	20	&	32.707	&	32.580	&	0.127	&	32.273	&	0.434	&	5000	\\
3	&	21	&	32.216	&	32.036	&	0.180	&	31.865	&	0.351	&	5000	\\
3	&	22	&	31.896	&	31.853	&	0.044	&	31.777	&	0.119	&	5000	\\
3	&	23	&	30.506	&	30.390	&	0.116	&	30.188	&	0.319	&	5000	\\
3	&	24	&	30.163	&	30.089	&	0.074	&	29.694	&	0.469	&	5000	\\
3	&	25	&	29.249	&	29.024	&	0.224	&	28.541	&	0.707	&	5000	\\
\hline\hline
4	&	5	&	75.522	&	75.522	&	0.001	&	73.410	&	2.113	&	4071	\\
4	&	6	&	70.529	&	70.528	&	0.001	&	70.528	&	0.001	&	91	\\
4	&	7	&	67.021	&	67.021	&	0.001	&	67.021	&	0.001	&	325	\\
4	&	8	&	65.530	&	65.530	&	0.001	&	64.688	&	0.842	&	3134	\\
4	&	9	&	64.262	&	64.261	&	0.001	&	64.261	&	0.001	&	1843	\\
4	&	10	&	64.262	&	64.261	&	0.001	&	64.261	&	0.001	&	803	\\
4	&	11	&	60.000	&	59.999	&	0.001	&	59.999	&	0.001	&	577	\\
4	&	12	&	60.000	&	59.999	&	0.001	&	59.999	&	0.001	&	146	\\
4	&	13	&	55.465	&	55.464	&	0.001	&	54.390	&	1.074	&	4629	\\
4	&	14	&	53.838	&	53.833	&	0.005	&	53.405	&	0.433	&	5000	\\
4	&	15	&	52.502	&	52.493	&	0.009	&	51.916	&	0.585	&	5000	\\
4	&	16	&	51.827	&	51.714	&	0.113	&	50.931	&	0.896	&	5000	\\
4	&	17	&	50.887	&	50.834	&	0.053	&	50.286	&	0.601	&	5000	\\
4	&	18	&	50.458	&	50.364	&	0.094	&	49.915	&	0.542	&	5000	\\
4	&	19	&	49.711	&	49.669	&	0.041	&	49.304	&	0.406	&	5000	\\
4	&	20	&	49.233	&	49.191	&	0.042	&	48.903	&	0.330	&	5000	\\
4	&	21	&	48.548	&	48.464	&	0.084	&	48.374	&	0.174	&	5000	\\
4	&	22	&	47.760	&	47.708	&	0.052	&	47.508	&	0.251	&	5000	\\
4	&	23	&	46.510	&	46.202	&	0.308	&	45.789	&	0.722	&	5000	\\
4	&	24	&	46.048	&	45.938	&	0.110	&	45.725	&	0.322	&	5000	\\
4	&	25	&	44.947	&	44.739	&	0.208	&	44.409	&	0.538	&	5000	\\
\hline\hline
5	&	6	&	78.463	&	78.463	&	0.001	&	77.359	&	1.104	&	3246	\\
5	&	7	&	73.369	&	73.368	&	0.001	&	73.368	&	0.001	&	1013	\\
5	&	8	&	70.804	&	70.803	&	0.001	&	70.604	&	0.200	&	5000	\\
5	&	9	&	70.529	&	70.528	&	0.001	&	69.576	&	0.953	&	2116	\\
5	&	10	&	70.529	&	70.528	&	0.001	&	67.033	&	3.496	&	3029	\\
5	&	11	&	67.254	&	67.254	&	0.001	&	66.015	&	1.239	&	4615	\\
5	&	12	&	67.021	&	66.486	&	0.535	&	65.661	&	1.361	&	5000	\\
5	&	13	&	65.732	&	65.720	&	0.012	&	65.435	&	0.297	&	5000	\\
5	&	14	&	65.724	&	65.723	&	0.001	&	65.637	&	0.087	&	3559	\\
5	&	15	&	65.530	&	65.492	&	0.038	&	65.443	&	0.088	&	5000	\\
5	&	16	&	63.435	&	63.434	&	0.001	&	63.434	&	0.001	&	940	\\
5	&	17	&	61.255	&	61.238	&	0.017	&	60.969	&	0.287	&	5000	\\
5	&	18	&	61.053	&	61.048	&	0.005	&	60.946	&	0.107	&	5000	\\
5	&	19	&	60.000	&	58.238	&	1.762	&	57.526	&	2.474	&	5000	\\
5	&	20	&	60.000	&	59.999	&	0.001	&	56.183	&	3.817	&	3290	\\
5	&	21	&	57.202	&	57.134	&	0.068	&	56.159	&	1.043	&	5000	\\
5	&	22	&	56.356	&	55.819	&	0.536	&	55.173	&	1.183	&	5000	\\
5	&	23	&	55.588	&	55.113	&	0.475	&	54.535	&	1.053	&	5000	\\
5	&	24	&	55.228	&	54.488	&	0.740	&	53.926	&	1.302	&	5000	\\
5	&	25	&	54.889	&	54.165	&	0.724	&	52.990	&	1.899	&	5000	\\
\end{supertabular}

\normalsize
\end{center}

\newpage

\begin{center}
\tablefirsthead{%
	\cline{3-5}
	\multicolumn{2}{c|}{} &
	\multicolumn{3}{|c|}{\textsc{Packing diameters (Degrees)}} \\
	\hline
	\multicolumn{1}{|c}{$d$} &
	\multicolumn{1}{c||}{$N$} &
	\multicolumn{1}{|c|}{\phantom{123} DHST} &
	\multicolumn{1}{|c||}{\phantom{123} Rankin} &
	Difference \\
	\hline\hline}
\tablehead{%
	\multicolumn{5}{l}{\textsl{\dots continued}} \\
	\multicolumn{5}{l}{} \\
	\cline{3-5}
	\multicolumn{2}{c|}{} &
	\multicolumn{3}{|c|}{\textsc{Packing diameters (Degrees)}} \\
	\hline
	\multicolumn{1}{|c}{$d$} &
	\multicolumn{1}{c||}{$N$} &
	\multicolumn{1}{|c|}{\phantom{123} DHST} &
	\multicolumn{1}{|c||}{\phantom{123} Rankin} &
	Difference \\
	\hline\hline}
\tabletail{%
	\hline
	\multicolumn{5}{r}{} \\
	\multicolumn{5}{r}{\textsl{continued\dots}} \\}
\tablelasttail{\hline}

\topcaption[Packing in complex projective spaces]{\textsc{Packing in complex projective spaces:}  We compare our best configurations (DHST) of $N$ points in the complex projective space $\CPspace{d -1}$ against the Rankin bound \eqref{eqn:rankin-chord}.  The packing diameter of an ensemble is measured as the acute angle (in degrees) between the closest pair of lines.  The final column shows how far our configurations fall short of the bound.} \label{tab:complex-projective}
%If the bound is met, the lines form an equiangular tight frame.

\small
\begin{supertabular}{|rr||r|r||r|}
2 &	3	&	60.00	&	60.00	&	0.00	\\
2 &	4	&	54.74	&	54.74	&	0.00	\\
2 &	5	&	45.00	&	52.24	&	7.24	\\
2 &	6	&	45.00	&	50.77	&	5.77	\\
2 &	7	&	38.93	&	49.80	&	10.86	\\
2 &	8	&	37.41	&	49.11	&	11.69	\\
%2 &	9	&	35.26	&	48.59	&	13.33	\\
%2 &	10	&	31.49	&	48.19	&	16.70	\\
\hline\hline
3 &	4	&	70.53	&	70.53	&	0.00	\\
3 &	5	&	64.00	&	65.91	&	1.90	\\
3 &	6	&	63.44	&	63.43	&	0.00	\\
3 &	7	&	61.87	&	61.87	&	0.00	\\
3 &	8	&	60.00	&	60.79	&	0.79	\\
3 &	9	&	60.00	&	60.00	&	0.00	\\
3 &	10	&	54.73	&	59.39	&	4.66	\\
3 &	11	&	54.73	&	58.91	&	4.18	\\
3 &	12	&	54.73	&	58.52	&	3.79	\\
3 &	13	&	51.32	&	58.19	&	6.88	\\
3 &	14	&	50.13	&	57.92	&	7.79	\\
3 &	15	&	49.53	&	57.69	&	8.15	\\
3 &	16	&	49.53	&	57.49	&	7.95	\\
3 &	17	&	49.10	&	57.31	&	8.21	\\
3 &	18	&	48.07	&	57.16	&	9.09	\\
3 &	19	&	47.02	&	57.02	&	10.00	\\
3 &	20	&	46.58	&	56.90	&	10.32	\\
\hline\hline
4 &	5	&	75.52	&	75.52	&	0.00	\\
4 &	6	&	70.88	&	71.57	&	0.68	\\
4 &	7	&	69.29	&	69.30	&	0.01	\\
4 &	8	&	67.78	&	67.79	&	0.01	\\
4 &	9	&	66.21	&	66.72	&	0.51	\\
4 &	10	&	65.71	&	65.91	&	0.19	\\
4 &	11	&	64.64	&	65.27	&	0.63	\\
4 &	12	&	64.24	&	64.76	&	0.52	\\
4 &	13	&	64.34	&	64.34	&	0.00	\\
4 &	14	&	63.43	&	63.99	&	0.56	\\
4 &	15	&	63.43	&	63.69	&	0.26	\\
4 &	16	&	63.43	&	63.43	&	0.00	\\
4 &	17	&	59.84	&	63.21	&	3.37	\\
4 &	18	&	59.89	&	63.02	&	3.12	\\
4 &	19	&	60.00	&	62.84	&	2.84	\\
4 &	20	&	57.76	&	62.69	&	4.93	\\
\hline\hline
5 &	6	&	78.46	&	78.46	&	0.00	\\
5 &	7	&	74.52	&	75.04	&	0.51	\\
5 &	8	&	72.81	&	72.98	&	0.16	\\
5 &	9	&	71.24	&	71.57	&	0.33	\\
5 &	10	&	70.51	&	70.53	&	0.02	\\
5 &	11	&	69.71	&	69.73	&	0.02	\\
5 &	12	&	68.89	&	69.10	&	0.21	\\
5 &	13	&	68.19	&	68.58	&	0.39	\\
5 &	14	&	67.66	&	68.15	&	0.50	\\
5 &	15	&	67.37	&	67.79	&	0.43	\\
5 &	16	&	66.68	&	67.48	&	0.80	\\
5 &	17	&	66.53	&	67.21	&	0.68	\\
5 &	18	&	65.87	&	66.98	&	1.11	\\
5 &	19	&	65.75	&	66.77	&	1.02	\\
5 &	20	&	65.77	&	66.59	&	0.82	\\
5 &	21	&	65.83	&	66.42	&	0.60	\\
5 &	22	&	65.87	&	66.27	&	0.40	\\
5 &	23	&	65.90	&	66.14	&	0.23	\\
5 &	24	&	65.91	&	66.02	&	0.11	\\
5 &	25	&	65.91	&	65.91	&	0.00	\\
\end{supertabular}
\normalsize

\end{center}

%%%%%%%%%%%%%%%%%%%%%%%%%%%%%%%%%%%%%%%%%%%%%%%%%%%%%%%%%%%%%%%%%%%%%
%%%%%%%%%%%%%%%%%%%%%%%%%%%%%%%%%%%%%%%%%%%%%%%%%%%%%%%%%%%%%%%%%%%%%

% Real versus complex projective packings

\newpage

\begin{figure}
\begin{center}

\caption[Real and complex projective packings]{\textsc{Real and Complex Projective Packings:} These three graphs compare the packing diameters attained by configurations in real and complex projective spaces with $d = 3, 4, 5$.  The circles indicate the best real packings obtained by Sloane and his colleagues \cite{Slo:Grassmannian-Web}.  The crosses indicate the best complex packings produced by the authors.  Rankin's upper bound \eqref{eqn:rankin-chord} is depicted in gray.  The dashed vertical line marks the largest number of real lines for which the Rankin bound is attainable, while the solid vertical line marks the maximum number of complex lines for which the Rankin bound is attainable.} \label{fig:proj-packings}

\includegraphics[width=\textwidth]{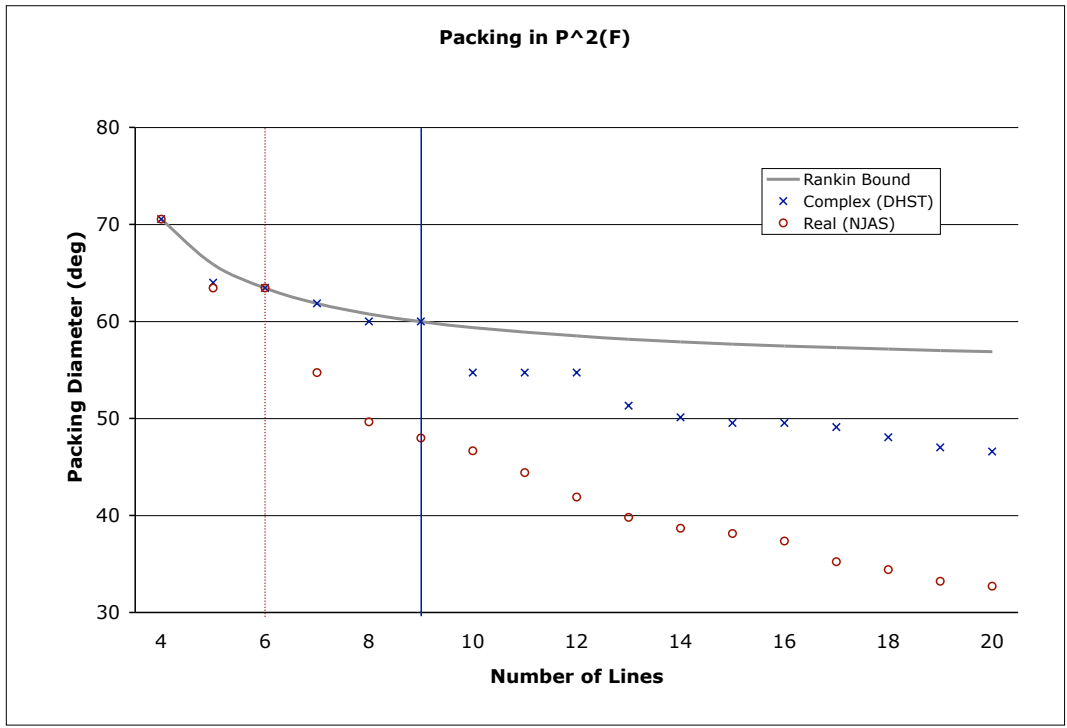}
\end{center}

\begin{flushright}
\textsl{continued\dots}
\end{flushright}

\end{figure}

\begin{figure}
\begin{flushleft}
\textsl{\dots continued}
\end{flushleft}

\vspace{-3pc}
\begin{center}
\includegraphics[width=\textwidth]{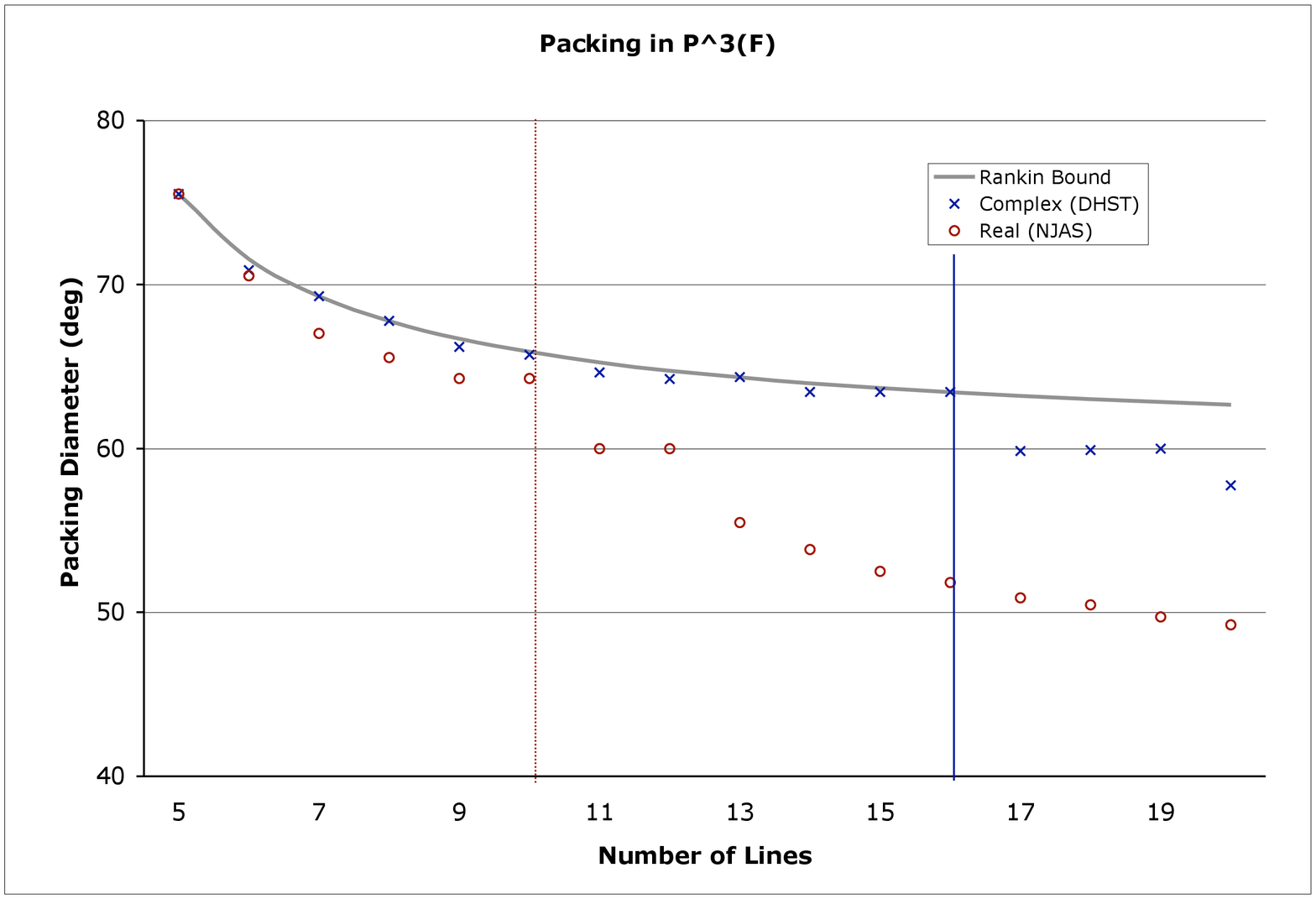} \\
\vspace{-4pc}
\includegraphics[width=\textwidth]{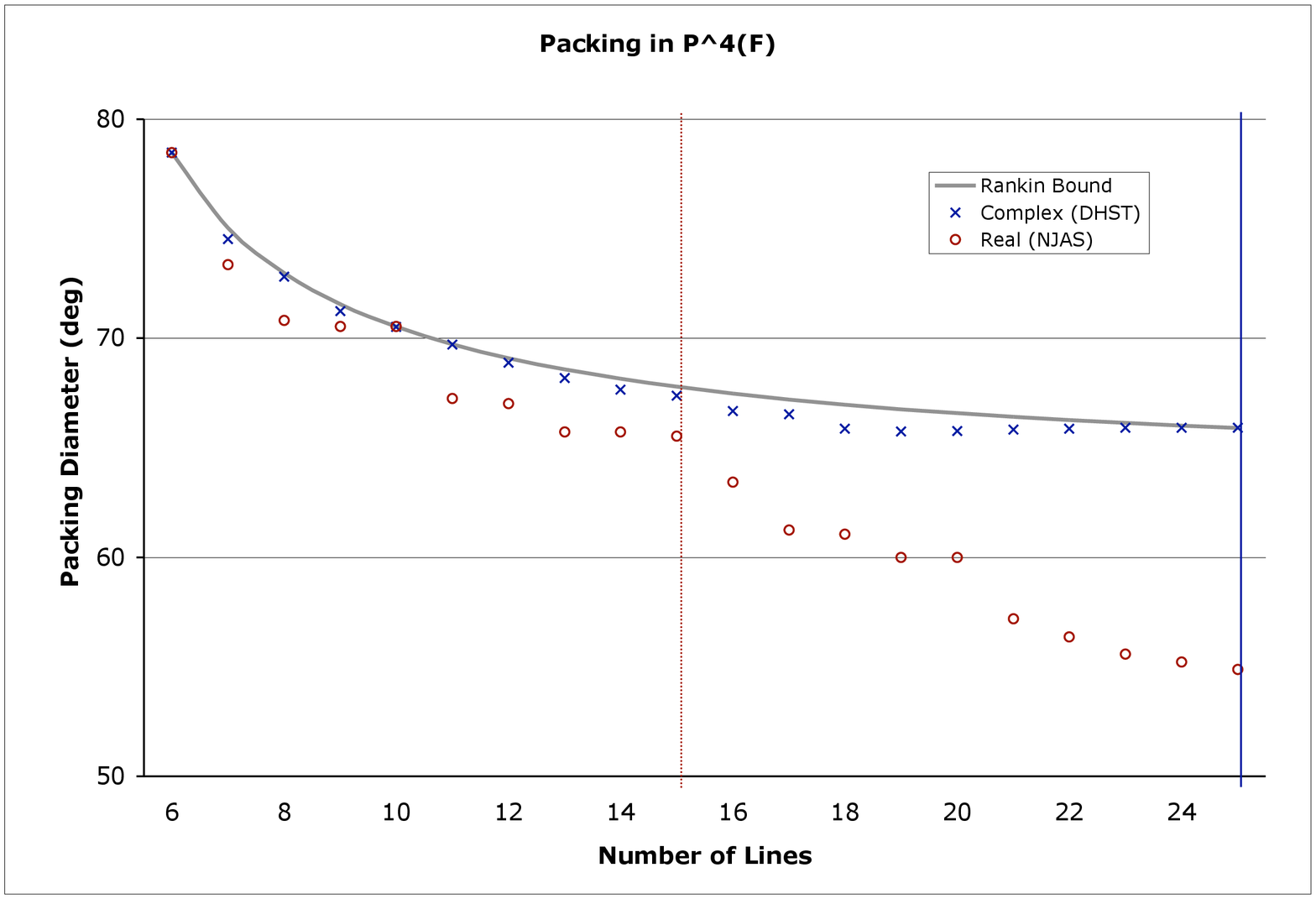}
\end{center}

\end{figure}

\cleardoublepage

\newpage

\begin{center}
\tablefirsthead{%
	\cline{4-6}
	\multicolumn{3}{c|}{} &
	\multicolumn{3}{|c|}{\textsc{Squared Packing diameters}} \\
	\hline
	\multicolumn{1}{|c}{$K$} &
	\multicolumn{1}{c}{$d$} &
	\multicolumn{1}{c||}{$N$} &
	\multicolumn{1}{|c|}{\phantom{123} DHST} &
	\multicolumn{1}{|c||}{\phantom{123} NJAS} &
	Difference \\
	\hline\hline}
\tablehead{%
	\multicolumn{6}{l}{\textsl{\dots continued}} \\
	\multicolumn{6}{l}{} \\
	\cline{4-6}
	\multicolumn{3}{c|}{} &
	\multicolumn{3}{|c|}{\textsc{Squared Packing diameters}} \\
	\hline
	\multicolumn{1}{|c}{$K$} &
	\multicolumn{1}{c}{$d$} &
	\multicolumn{1}{c||}{$N$} &
	\multicolumn{1}{|c|}{\phantom{123} DHST} &
	\multicolumn{1}{|c||}{\phantom{123} NJAS} &
	Difference \\
	\hline\hline}
\tabletail{%
	\hline
	\multicolumn{6}{r}{} \\
	\multicolumn{6}{r}{\textsl{continued\dots}} \\}
\tablelasttail{\hline}

\topcaption[Packing in real Grassmannians with chordal distance]{\textsc{Packing in real Grassmannians with chordal distance:}
We compare our best configurations (DHST) of $N$ points in $\RGspace{K}{d}$ against the best packings (NJAS) reported in \cite{Slo:Grassmannian-Web}.  The squared packing diameter is the squared chordal distance \eqref{eqn:chordal-distance} between the closest pair of subspaces.  The last column lists the difference between the columns (NJAS) and (DHST).}
\label{tab:real-chordal}

\small
\begin{supertabular}{|rrr||r|r||r|}
2 & 4 &	3	&	1.5000	&	1.5000	&	0.0000	\\
2 & 4 &	4	&	1.3333	&	1.3333	&	0.0000	\\
2 & 4 &	5	&	1.2500	&	1.2500	&	0.0000	\\
2 & 4 &	6	&	1.2000	&	1.2000	&	0.0000	\\
2 & 4 &	7	&	1.1656	&	1.1667	&	0.0011	\\
2 & 4 &	8	&	1.1423	&	1.1429	&	0.0005	\\
2 & 4 &	9	&	1.1226	&	1.1231	&	0.0004	\\
2 & 4 &	10	&	1.1111	&	1.1111	&	0.0000	\\
2 & 4 &	11	&	0.9981	&	1.0000	&	0.0019	\\
2 & 4 &	12	&	0.9990	&	1.0000	&	0.0010	\\
2 & 4 &	13	&	0.9996	&	1.0000	&	0.0004	\\
2 & 4 &	14	&	1.0000	&	1.0000	&	0.0000	\\
2 & 4 &	15	&	1.0000	&	1.0000	&	0.0000	\\
2 & 4 &	16	&	0.9999	&	1.0000	&	0.0001	\\
2 & 4 &	17	&	1.0000	&	1.0000	&	0.0000	\\
2 & 4 &	18	&	0.9992	&	1.0000	&	0.0008	\\
2 & 4 &	19	&	0.8873	&	0.9091	&	0.0218	\\
2 & 4 &	20	&	0.8225	&	0.9091	&	0.0866	\\
\hline\hline
2 & 5 &	3	&	1.7500	&	1.7500	&	0.0000	\\
2 & 5 &	4	&	1.6000	&	1.6000	&	0.0000	\\
2 & 5 &	5	&	1.5000	&	1.5000	&	0.0000	\\
2 & 5 &	6	&	1.4400	&	1.4400	&	0.0000	\\
2 & 5 &	7	&	1.4000	&	1.4000	&	0.0000	\\
2 & 5 &	8	&	1.3712	&	1.3714	&	0.0002	\\
2 & 5 &	9	&	1.3464	&	1.3500	&	0.0036	\\
2 & 5 &	10	&	1.3307	&	1.3333	&	0.0026	\\
2 & 5 &	11	&	1.3069	&	1.3200	&	0.0131	\\
2 & 5 &	12	&	1.2973	&	1.3064	&	0.0091	\\
2 & 5 &	13	&	1.2850	&	1.2942	&	0.0092	\\
2 & 5 &	14	&	1.2734	&	1.2790	&	0.0056	\\
2 & 5 &	15	&	1.2632	&	1.2707	&	0.0075	\\
2 & 5 &	16	&	1.1838	&	1.2000	&	0.0162	\\
2 & 5 &	17	&	1.1620	&	1.2000	&	0.0380	\\
2 & 5 &	18	&	1.1589	&	1.1909	&	0.0319	\\
2 & 5 &	19	&	1.1290	&	1.1761	&	0.0472	\\
2 & 5 &	20	&	1.0845	&	1.1619	&	0.0775	\\
\end{supertabular}
\normalsize

\end{center}

%%%%%%%%%%%%%%%%%%%%%%%%%%%%%%%%%%%%%%%%%%%%%%%%%%%%%%%%%%%%%%%%%%%%%
%%%%%%%%%%%%%%%%%%%%%%%%%%%%%%%%%%%%%%%%%%%%%%%%%%%%%%%%%%%%%%%%%%%%%

% Table of complex Grassmannian packings, chordal distance

% \label{tab:complex-chordal}

\newpage

\begin{center}
\tablefirsthead{%
	\cline{4-6}
	\multicolumn{3}{c|}{} &
	\multicolumn{3}{|c|}{\textsc{Squared Packing diameters}} \\
	\hline
	\multicolumn{1}{|c}{$K$} &
	\multicolumn{1}{c}{$d$} &
	\multicolumn{1}{c||}{$N$} &
	\multicolumn{1}{|c|}{\phantom{123} DHST} &
	\multicolumn{1}{|c||}{\phantom{123} Rankin} &
	Difference \\
	\hline\hline}
\tablehead{%
	\multicolumn{6}{l}{\textsl{\dots continued}} \\
	\multicolumn{6}{l}{} \\
	\cline{4-6}
	\multicolumn{3}{c|}{} &
	\multicolumn{3}{|c|}{\textsc{Squared Packing diameters}} \\
	\hline
	\multicolumn{1}{|c}{$K$} &
	\multicolumn{1}{c}{$d$} &
	\multicolumn{1}{c||}{$N$} &
	\multicolumn{1}{|c|}{\phantom{123} DHST} &
	\multicolumn{1}{|c||}{\phantom{123} Rankin} &
	Difference \\
	\hline\hline}
\tabletail{%
	\hline
	\multicolumn{6}{r}{} \\
	\multicolumn{6}{r}{\textsl{continued\dots}} \\}
\tablelasttail{\hline}

\topcaption[Packing in complex Grassmannians with chordal distance]{\textsc{Packing in complex Grassmannians with chordal distance:}
We compare our best configurations (DHST) of $N$ points in $\CGspace{K}{d}$ against the Rankin bound, equation \eqref{eqn:rankin-chord}.  The squared packing diameter is calculated as the squared chordal distance \eqref{eqn:chordal-distance} between the closest pair of subspaces.  The final column shows how much the computed ensemble declines from the Rankin bound.  When the bound is met, all pairs of subspaces are equidistant.} \label{tab:complex-chordal}

\small
\begin{supertabular}{|rrr||r|r||r|}
2 & 4 &	3	&	1.5000	&	1.5000	&	0.0000	\\
2 & 4 &	4	&	1.3333	&	1.3333	&	0.0000	\\
2 & 4 &	5	&	1.2500	&	1.2500	&	0.0000	\\
2 & 4 &	6	&	1.2000	&	1.2000	&	0.0000	\\
2 & 4 &	7	&	1.1667	&	1.1667	&	0.0000	\\
2 & 4 &	8	&	1.1429	&	1.1429	&	0.0000	\\
2 & 4 &	9	&	1.1250	&	1.1250	&	0.0000	\\
2 & 4 &	10	&	1.1111	&	1.1111	&	0.0000	\\
2 & 4 &	11	&	1.0999	&	1.1000	&	0.0001	\\
2 & 4 &	12	&	1.0906	&	1.0909	&	0.0003	\\
2 & 4 &	13	&	1.0758	&	1.0833	&	0.0076	\\
2 & 4 &	14	&	1.0741	&	1.0769	&	0.0029	\\
2 & 4 &	15	&	1.0698	&	1.0714	&	0.0016	\\
2 & 4 &	16	&	1.0658	&	1.0667	&	0.0009	\\
2 & 4 &	17	&	0.9975	&	1.0625	&	0.0650	\\
2 & 4 &	18	&	0.9934	&	1.0588	&	0.0654	\\
2 & 4 &	19	&	0.9868	&	1.0556	&	0.0688	\\
2 & 4 &	20	&	0.9956	&	1.0526	&	0.0571	\\
\hline\hline
2 & 5 &	3	&	1.7500	&	1.8000	&	0.0500	\\
2 & 5 &	4	&	1.6000	&	1.6000	&	0.0000	\\
2 & 5 &	5	&	1.5000	&	1.5000	&	0.0000	\\
2 & 5 &	6	&	1.4400	&	1.4400	&	0.0000	\\
2 & 5 &	7	&	1.4000	&	1.4000	&	0.0000	\\
2 & 5 &	8	&	1.3714	&	1.3714	&	0.0000	\\
2 & 5 &	9	&	1.3500	&	1.3500	&	0.0000	\\
2 & 5 &	10	&	1.3333	&	1.3333	&	0.0000	\\
2 & 5 &	11	&	1.3200	&	1.3200	&	0.0000	\\
2 & 5 &	12	&	1.3090	&	1.3091	&	0.0001	\\
2 & 5 &	13	&	1.3000	&	1.3000	&	0.0000	\\
2 & 5 &	14	&	1.2923	&	1.2923	&	0.0000	\\
2 & 5 &	15	&	1.2857	&	1.2857	&	0.0000	\\
2 & 5 &	16	&	1.2799	&	1.2800	&	0.0001	\\
2 & 5 &	17	&	1.2744	&	1.2750	&	0.0006	\\
2 & 5 &	18	&	1.2686	&	1.2706	&	0.0020	\\
2 & 5 &	19	&	1.2630	&	1.2667	&	0.0037	\\
2 & 5 &	20	&	1.2576	&	1.2632	&	0.0056	\\
\hline\hline
2 & 6 &	4	&	1.7778	&	1.7778	&	0.0000	\\
2 & 6 &	5	&	1.6667	&	1.6667	&	0.0000	\\
2 & 6 &	6	&	1.6000	&	1.6000	&	0.0000	\\
2 & 6 &	7	&	1.5556	&	1.5556	&	0.0000	\\
2 & 6 &	8	&	1.5238	&	1.5238	&	0.0000	\\
2 & 6 &	9	&	1.5000	&	1.5000	&	0.0000	\\
2 & 6 &	10	&	1.4815	&	1.4815	&	0.0000	\\
2 & 6 &	11	&	1.4667	&	1.4667	&	0.0000	\\
2 & 6 &	12	&	1.4545	&	1.4545	&	0.0000	\\
2 & 6 &	13	&	1.4444	&	1.4444	&	0.0000	\\
2 & 6 &	14	&	1.4359	&	1.4359	&	0.0000	\\
2 & 6 &	15	&	1.4286	&	1.4286	&	0.0000	\\
2 & 6 &	16	&	1.4221	&	1.4222	&	0.0001	\\
2 & 6 &	17	&	1.4166	&	1.4167	&	0.0000	\\
2 & 6 &	18	&	1.4118	&	1.4118	&	0.0000	\\
2 & 6 &	19	&	1.4074	&	1.4074	&	0.0000	\\
2 & 6 &	20	&	1.4034	&	1.4035	&	0.0001	\\
2 & 6 &	21	&	1.3999	&	1.4000	&	0.0001	\\
2 & 6 &	22	&	1.3968	&	1.3968	&	0.0001	\\
2 & 6 &	23	&	1.3923	&	1.3939	&	0.0017	\\
2 & 6 &	24	&	1.3886	&	1.3913	&	0.0028	\\
2 & 6 &	25	&	1.3862	&	1.3889	&	0.0027	\\
\hline\hline
3 & 6 &	3	&	2.2500	&	2.2500	&	0.0000	\\
3 & 6 &	4	&	2.0000	&	2.0000	&	0.0000	\\
3 & 6 &	5	&	1.8750	&	1.8750	&	0.0000	\\
3 & 6 &	6	&	1.8000	&	1.8000	&	0.0000	\\
3 & 6 &	7	&	1.7500	&	1.7500	&	0.0000	\\
3 & 6 &	8	&	1.7143	&	1.7143	&	0.0000	\\
3 & 6 &	9	&	1.6875	&	1.6875	&	0.0000	\\
3 & 6 &	10	&	1.6667	&	1.6667	&	0.0000	\\
3 & 6 &	11	&	1.6500	&	1.6500	&	0.0000	\\
3 & 6 &	12	&	1.6363	&	1.6364	&	0.0001	\\
3 & 6 &	13	&	1.6249	&	1.6250	&	0.0001	\\
3 & 6 &	14	&	1.6153	&	1.6154	&	0.0000	\\
3 & 6 &	15	&	1.6071	&	1.6071	&	0.0000	\\
3 & 6 &	16	&	1.5999	&	1.6000	&	0.0001	\\
3 & 6 &	17	&	1.5936	&	1.5938	&	0.0001	\\
3 & 6 &	18	&	1.5879	&	1.5882	&	0.0003	\\
3 & 6 &	19	&	1.5829	&	1.5833	&	0.0004	\\
3 & 6 &	20	&	1.5786	&	1.5789	&	0.0004	\\
3 & 6 &	21	&	1.5738	&	1.5750	&	0.0012	\\
3 & 6 &	22	&	1.5687	&	1.5714	&	0.0028	\\
3 & 6 &	23	&	1.5611	&	1.5682	&	0.0070	\\
3 & 6 &	24	&	1.5599	&	1.5652	&	0.0053	\\
3 & 6 &	25	&	1.5558	&	1.5625	&	0.0067	\\
3 & 6 &	26	&	1.5542	&	1.5600	&	0.0058	\\
3 & 6 &	27	&	1.5507	&	1.5577	&	0.0070	\\
3 & 6 &	28	&	1.5502	&	1.5556	&	0.0054	\\
3 & 6 &	29	&	1.5443	&	1.5536	&	0.0092	\\
3 & 6 &	30	&	1.5316	&	1.5517	&	0.0201	\\
3 & 6 &	31	&	1.5283	&	1.5500	&	0.0217	\\
3 & 6 &	32	&	1.5247	&	1.5484	&	0.0237	\\
3 & 6 &	33	&	1.5162	&	1.5469	&	0.0307	\\
3 & 6 &	34	&	1.5180	&	1.5455	&	0.0274	\\
3 & 6 &	35	&	1.5141	&	1.5441	&	0.0300	\\
3 & 6 &	36	&	1.5091	&	1.5429	&	0.0338	\\
\end{supertabular}
\normalsize

\end{center}

%%%%%%%%%%%%%%%%%%%%%%%%%%%%%%%%%%%%%%%%%%%%%%%%%%%%%%%%%%%%%%%%%%%%%
%%%%%%%%%%%%%%%%%%%%%%%%%%%%%%%%%%%%%%%%%%%%%%%%%%%%%%%%%%%%%%%%%%%%%

% Real versus complex Grassmannian packings, chordal distance

\newpage

\begin{figure}
\caption[Packing in Grassmannians with chordal distance]{\textsc{Packing in Grassmannians with chordal distance:}  This figure shows the packing diameters of $N$ points in the Grassmannian $\FGspace{K}{d}$ equipped with the chordal distance.  The circles indicate the best real packings ($\mathbb{F} = \Rspace{}$) obtained by Sloane and his colleagues \cite{Slo:Grassmannian-Web}.  The crosses indicate the best complex packings ($\mathbb{F} = \Cspace{}$) produced by the authors.  Rankin's upper bound \eqref{eqn:rankin-chord} appears in gray.  The dashed vertical line marks the largest number of real subspaces for which the Rankin bound is attainable, while the solid vertical line marks the maximum number of complex subspaces for which the Rankin bound is attainable.} \label{fig:chord-packings}

\begin{center}
\includegraphics[width=\textwidth]{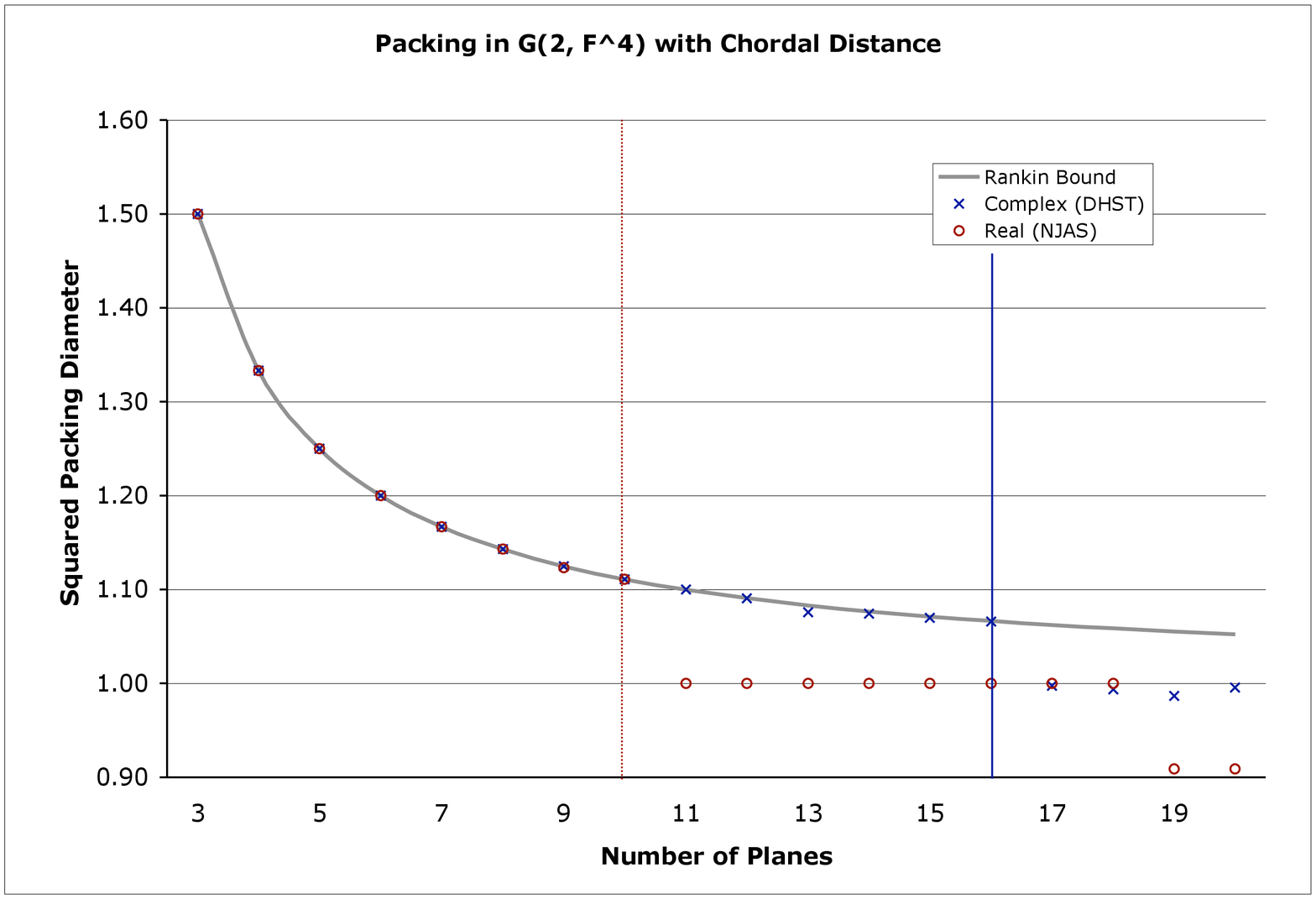}
\end{center}

\begin{flushright}
\textsl{continued\dots}
\end{flushright}
\end{figure}

\newpage

\begin{figure}
\begin{flushleft}
\textsl{\dots continued}
\end{flushleft}

\vspace{-3pc}
\begin{center}
\includegraphics[width=\textwidth]{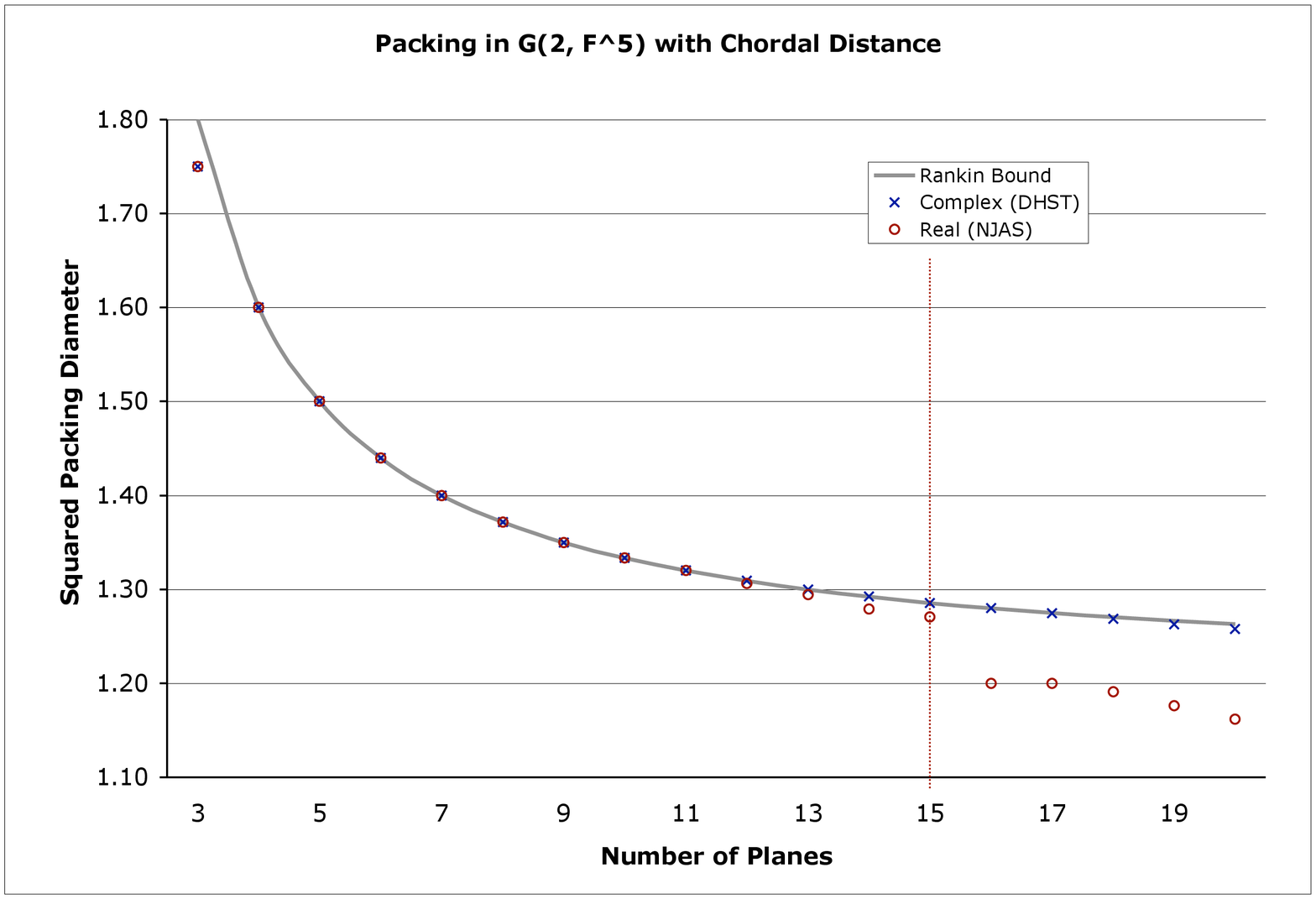} \\
\vspace{-4pc}
\includegraphics[width=\textwidth]{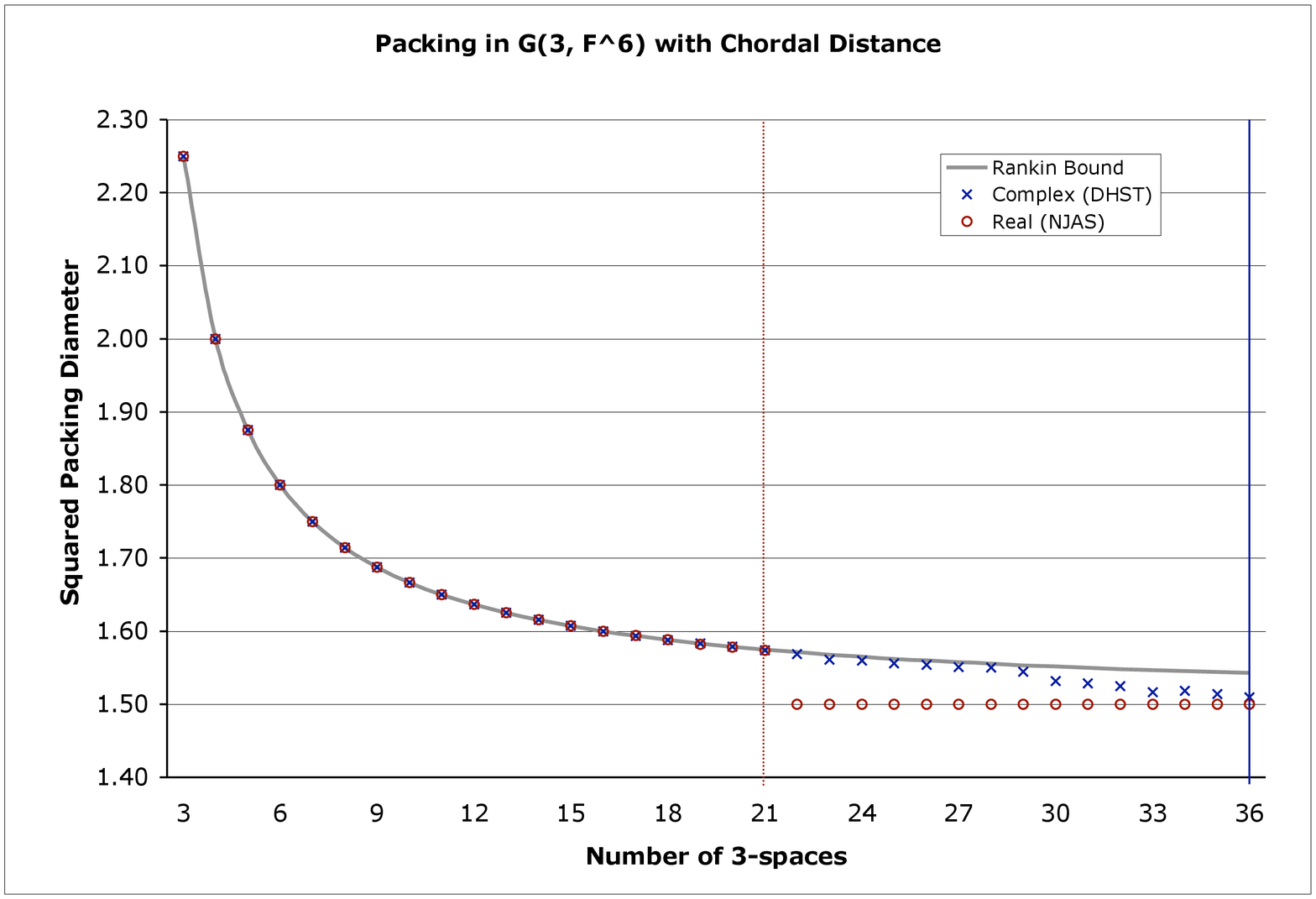}
\end{center}

\end{figure}

\clearpage

%%%%%%%%%%%%%%%%%%%%%%%%%%%%%%%%%%%%%%%%%%%%%%%%%%%%%%%%%%%%%%%%%%%%%
%%%%%%%%%%%%%%%%%%%%%%%%%%%%%%%%%%%%%%%%%%%%%%%%%%%%%%%%%%%%%%%%%%%%%

% Table of Grassmannian packings, projection distance

\newpage

\begin{center}
\tablefirsthead{%
	\cline{4-8}
	\multicolumn{3}{c|}{} &
	\multicolumn{5}{|c|}{\textsc{Squared Packing diameters}} \\
	\hline
	$d$ & $K$ & $N$ & Rankin &
	\multicolumn{1}{|c}{\phantom{123} $\Rspace{}$ \phantom{123}} & Difference &
	\multicolumn{1}{|c}{\phantom{123} $\Cspace{}$ \phantom{123}} & Difference \\
	\hline\hline}
\tablehead{%
	\multicolumn{8}{l}{\textsl{\dots continued}} \\
	\multicolumn{8}{l}{} \\
	\cline{4-8}
	\multicolumn{3}{c|}{} &
	\multicolumn{5}{|c|}{\textsc{Squared Packing diameters}} \\
	\hline
	$d$ & $K$ & $N$ & Rankin &
	\multicolumn{1}{|c}{\phantom{123} $\Rspace{}$ \phantom{123}} & Difference &
	\multicolumn{1}{|c}{\phantom{123} $\Cspace{}$ \phantom{123}} & Difference \\
	\hline\hline}
\tabletail{%
	\hline
	\multicolumn{8}{r}{} \\
	\multicolumn{8}{r}{\textsl{continued\dots}} \\}
\tablelasttail{\hline}

\topcaption[Packing in Grassmannians with spectral distance]{\textsc{Packing in Grassmannians with spectral distance:}  We compare our best real ($\mathbb{F} = \Rspace{}$) and complex ($\mathbb{F} = \Cspace{}$) packings in $\FGspace{K}{d}$ against the Rankin bound, equation \eqref{eqn:rankin-spec}.  The squared packing diameter of a configuration is the squared spectral distance \eqref{eqn:spectral-distance} between the closest pair of subspaces.  When the Rankin bound is met, all pairs of subspaces are equi-isoclinic.  The algorithm failed to produce any configurations of 8 or 9 subspaces in $\RGspace{3}{6}$ with nontrivial packing diameters.} \label{tab:projection-dist}

\small
\begin{supertabular}{|rrr|r|rr|rr|}
4	&	2	&	3	&	0.7500	&	0.7500	&	0.0000	&	0.7500	&	0.0000	\\
4	&	2	&	4	&	0.6667	&	0.6667	&	0.0000	&	0.6667	&	0.0000	\\
4	&	2	&	5	&	0.6250	&	0.5000	&	0.1250	&	0.6250	&	0.0000	\\
4	&	2	&	6	&	0.6000	&	0.4286	&	0.1714	&	0.6000	&	0.0000	\\
4	&	2	&	7	&	0.5833	&	0.3122	&	0.2712	&	0.5000	&	0.0833	\\
4	&	2	&	8	&	0.5714	&	0.2851	&	0.2863	&	0.4374	&	0.1340	\\
4	&	2	&	9	&	0.5625	&	0.2544	&	0.3081	&	0.4363	&	0.1262	\\
4	&	2	&	10	&	0.5556	&	0.2606	&	0.2950	&	0.4375	&	0.1181	\\
\hline\hline
5	&	2	&	3	&	0.9000	&	0.7500	&	0.1500	&	0.7500	&	0.1500	\\
5	&	2	&	4	&	0.8000	&	0.7500	&	0.0500	&	0.7500	&	0.0500	\\
5	&	2	&	5	&	0.7500	&	0.6700	&	0.0800	&	0.7497	&	0.0003	\\
5	&	2	&	6	&	0.7200	&	0.6014	&	0.1186	&	0.6637	&	0.0563	\\
5	&	2	&	7	&	0.7000	&	0.5596	&	0.1404	&	0.6667	&	0.0333	\\
5	&	2	&	8	&	0.6857	&	0.4991	&	0.1867	&	0.6060	&	0.0798	\\
5	&	2	&	9	&	0.6750	&	0.4590	&	0.2160	&	0.5821	&	0.0929	\\
5	&	2	&	10	&	0.6667	&	0.4615	&	0.2052	&	0.5196	&	0.1470	\\
\hline\hline
6	&	2	&	4	&	0.8889	&	0.8889	&	0.0000	&	0.8889	&	0.0000	\\
6	&	2	&	5	&	0.8333	&	0.7999	&	0.0335	&	0.8333	&	0.0000	\\
6	&	2	&	6	&	0.8000	&	0.8000	&	0.0000	&	0.8000	&	0.0000	\\
6	&	2	&	7	&	0.7778	&	0.7500	&	0.0278	&	0.7778	&	0.0000	\\
6	&	2	&	8	&	0.7619	&	0.7191	&	0.0428	&	0.7597	&	0.0022	\\
6	&	2	&	9	&	0.7500	&	0.6399	&	0.1101	&	0.7500	&	0.0000	\\
6	&	2	&	10	&	0.7407	&	0.6344	&	0.1064	&	0.7407	&	0.0000	\\
6	&	2	&	11	&	0.7333	&	0.6376	&	0.0958	&	0.7333	&	0.0000	\\
6	&	2	&	12	&	0.7273	&	0.6214	&	0.1059	&	0.7273	&	0.0000	\\
\hline
6	&	3	&	3	&	0.7500	&	0.7500	&	0.0000	&	0.7500	&	0.0000	\\
6	&	3	&	4	&	0.6667	&	0.5000	&	0.1667	&	0.6667	&	0.0000	\\
6	&	3	&	5	&	0.6250	&	0.4618	&	0.1632	&	0.4999	&	0.1251	\\
6	&	3	&	6	&	0.6000	&	0.4238	&	0.1762	&	0.5000	&	0.1000	\\
6	&	3	&	7	&	0.5833	&	0.3590	&	0.2244	&	0.4408	&	0.1426	\\
6	&	3	&	8	&	0.5714	&	---	&	---	&	0.4413	&	0.1301	\\
6	&	3	&	9	&	0.5625	&	---	&	---	&	0.3258	&	0.2367	\\
\end{supertabular}
\normalsize

\end{center}

\newpage

\begin{figure}
\caption[Packing in Grassmannians with spectral distance]{\textsc{Packing in Grassmannians with spectral distance:} This figure shows the packing diameters of $N$ points in the Grassmannian $\FGspace{K}{d}$ equipped with the spectral distance.  The circles indicate the best real packings ($\mathbb{F} = \Rspace{}$) obtained by the authors, while the crosses indicate the best complex packings ($\mathbb{F} = \Cspace{}$) obtained.  The Rankin bound \eqref{eqn:rankin-spec} is depicted in gray.  The dashed vertical line marks an upper bound on largest number of real subspaces for which the Rankin bound is attainable according to Theorem \ref{thm:ls-isoclinic}.} \label{fig:spec-packings}

\begin{center}
\includegraphics[width=\textwidth]{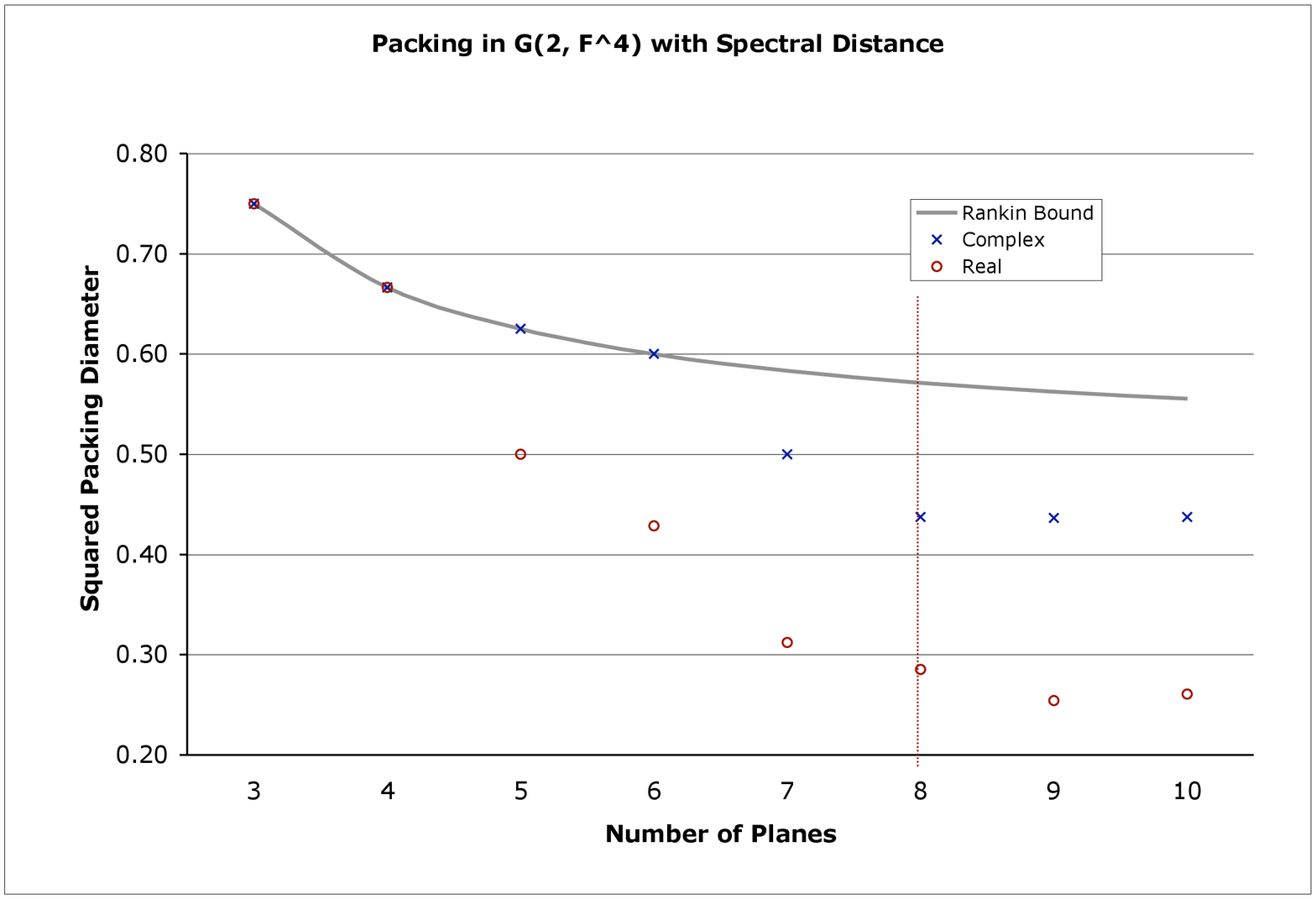}
\end{center}

\begin{flushright}
\textsl{continued\dots}
\end{flushright}

\end{figure}

\newpage

\begin{figure}
\begin{flushleft}
\textsl{\dots continued}
\end{flushleft}

\vspace{-3pc}
\begin{center}
\includegraphics[width=\textwidth]{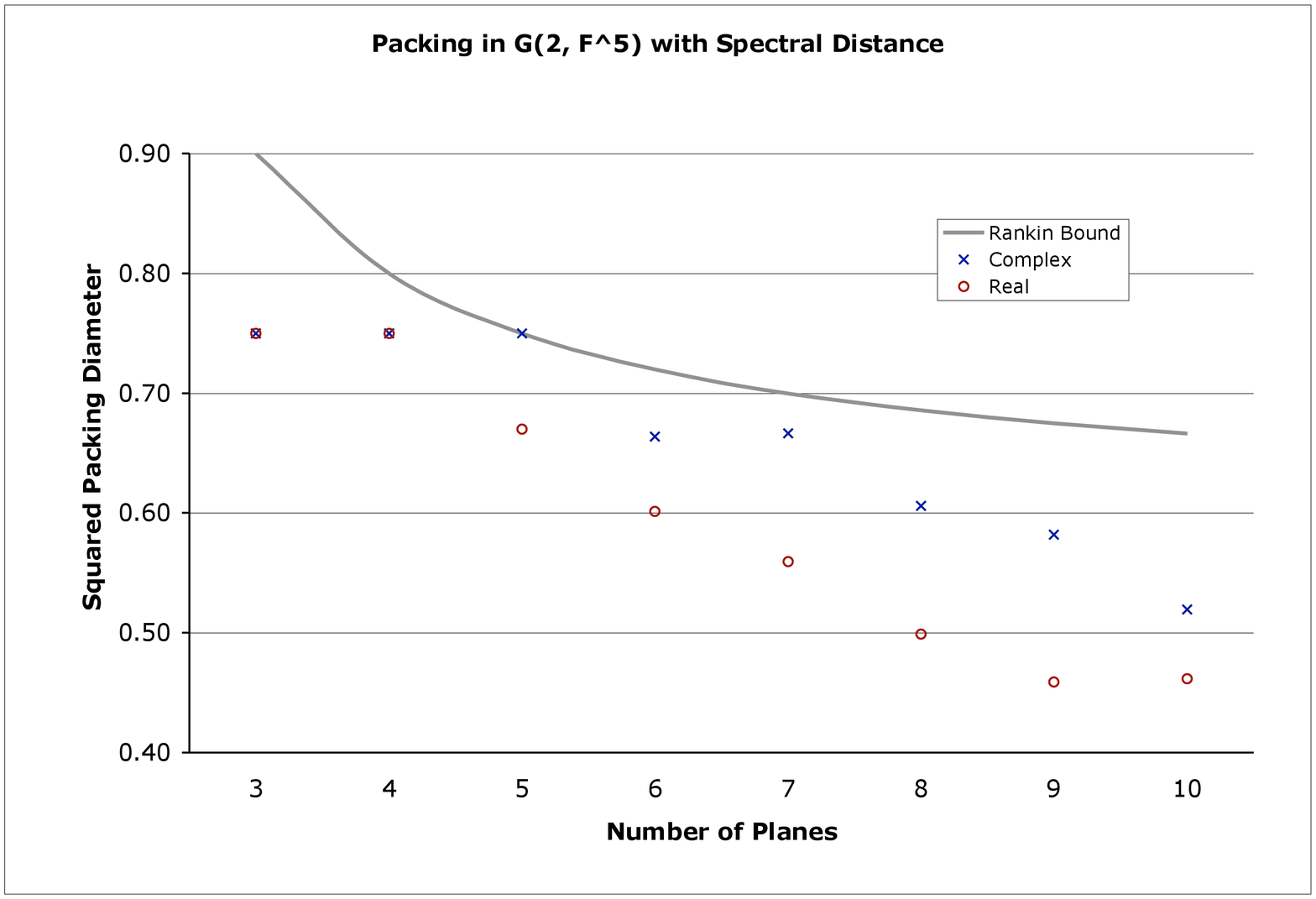} \\
\vspace{-4pc}
\includegraphics[width=\textwidth]{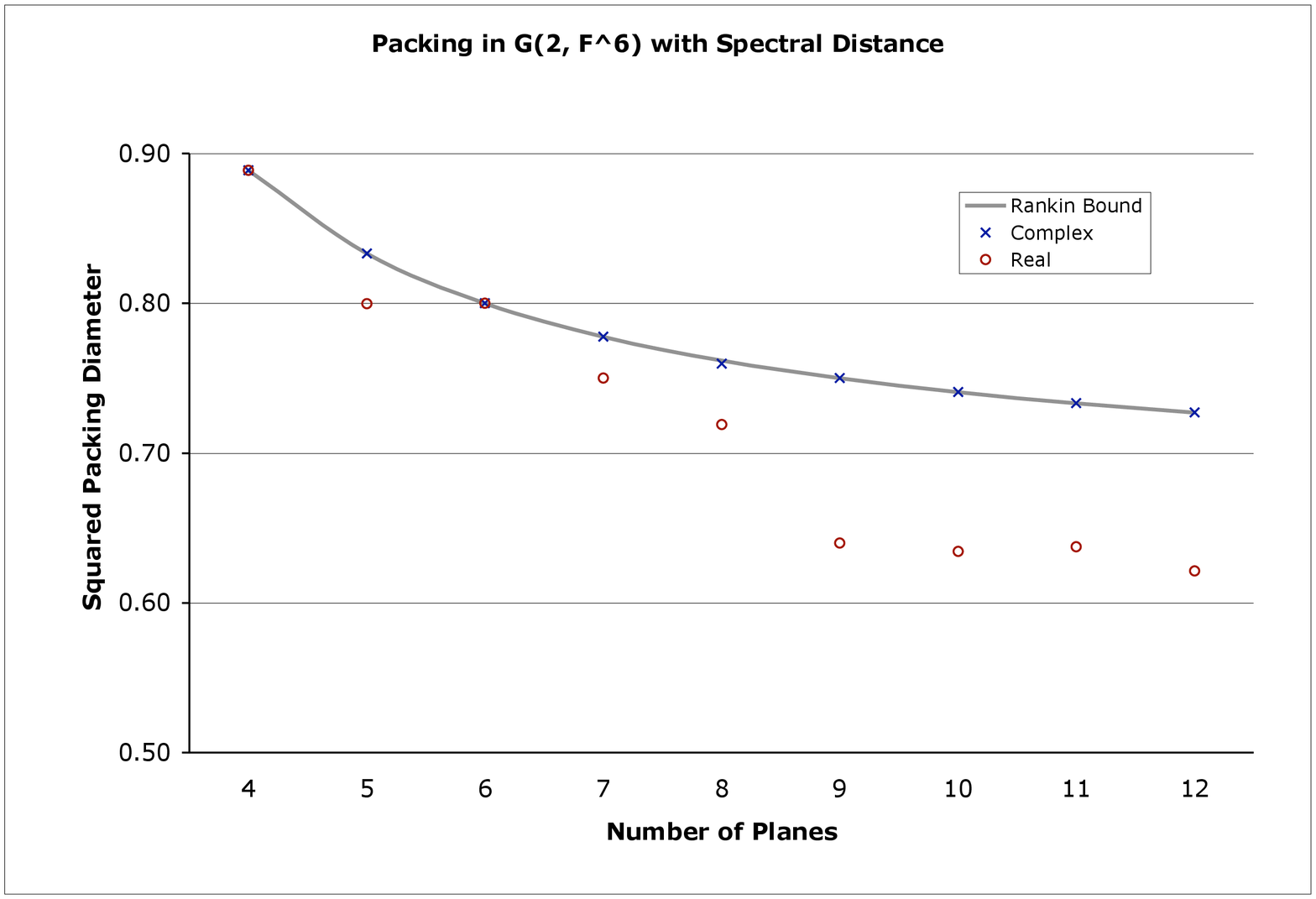} \\
\end{center}

\end{figure}

\clearpage

%%%%%%%%%%%%%%%%%%%%%%%%%%%%%%%%%%%%%%%%%%%%%%%%%%%%%%%%%%%%%%%%%%%%%
%%%%%%%%%%%%%%%%%%%%%%%%%%%%%%%%%%%%%%%%%%%%%%%%%%%%%%%%%%%%%%%%%%%%%

% Table of Grassmannian packings, Fubini--Study distance

\newpage

\begin{center}
\tablefirsthead{%
	\cline{4-5}
	\multicolumn{3}{c|}{} &
	\multicolumn{2}{|c|}{\textsc{Squared Packing diameters}} \\
	\hline
	$d$ & $K$ & $N$ &
	\multicolumn{1}{|c|}{\phantom{1234} $\Rspace{}$ \phantom{1234}} &
	\multicolumn{1}{|c|}{$\Cspace{}$} \\
	\hline\hline}
%\tablehead{%
%	\multicolumn{5}{l}{\textsl{\dots continued}} \\
%	\multicolumn{5}{l}{} \\
%	\cline{4-5}
%	\multicolumn{3}{c|}{} &
%	\multicolumn{2}{|c|}{\textsc{Squared Packing diameters}} \\
%	\hline
%	$d$ & $K$ & $N$ &
%	\phantom{12345} $\Rspace{}$ &
%	\phantom{12345} $\Cspace{}$ \\
%	\hline\hline}
\tabletail{%
	\hline
	\multicolumn{5}{r}{} \\
	\multicolumn{5}{r}{\textsl{continued\dots}} \\}
\tablelasttail{\hline}

\topcaption[Packing in Grassmannians with Fubini--Study distance]{\textsc{Packing in Grassmannians with Fubini--Study distance:}  Our best real packings ($\mathbb{F} = \Rspace{}$) compared with our best complex packings ($\mathbb{F} = \Cspace{}$) in the space $\FGspace{K}{d}$.  The packing diameter of a configuration is the Fubini--Study distance \eqref{eqn:fs-distance} between the closest pair of subspaces.  Note that we have scaled the distance by $2/\pi$ so that it ranges between zero and one.} \label{tab:fubini-study}

\small
\begin{supertabular}{|rrr||r|r|}
2 & 4 &	3	&	1.0000	&	1.0000	\\
2 & 4 &	4	&	1.0000	&	1.0000	\\
2 & 4 &	5	&	1.0000	&	1.0000	\\
2 & 4 &	6	&	1.0000	&	1.0000	\\
2 & 4 &	7	&	0.8933	&	0.8933	\\
2 & 4 &	8	&	0.8447	&	0.8559	\\
2 & 4 &	9	&	0.8196	&	0.8325	\\
2 & 4 &	10	&	0.8176	&	0.8216	\\
2 & 4 &	11	&	0.7818	&	0.8105	\\
2 & 4 &	12	&	0.7770	&	0.8033	\\
\hline\hline
2 & 5 &	3	&	1.0000	&	1.0000	\\
2 & 5 &	4	&	1.0000	&	1.0000	\\
2 & 5 &	5	&	1.0000	&	1.0000	\\
2 & 5 &	6	&	0.9999	&	1.0000	\\
2 & 5 &	7	&	1.0000	&	0.9999	\\
2 & 5 &	8	&	1.0000	&	0.9999	\\
2 & 5 &	9	&	1.0000	&	1.0000	\\
2 & 5 &	10	&	0.9998	&	1.0000	\\
2 & 5 &	11	&	0.9359	&	0.9349	\\
2 & 5 &	12	&	0.9027	&	0.9022	\\
\hline
\end{supertabular}
\normalsize

\end{center}

%%%%%%%%%%%%%%%%%%%%%%%%%%%%%%%%%%%%%%%%%%%%%%%%%%%%%%%%%%%%%%%%%%%%%
%%%%%%%%%%%%%%%%%%%%%%%%%%%%%%%%%%%%%%%%%%%%%%%%%%%%%%%%%%%%%%%%%%%%%

% Real versus complex Grassmannian packings, Fubini--Study distance

\newpage

\begin{figure}
\caption[Packing in Grassmannians with Fubini--Study distance]{\textsc{Packing in Grassmannians with Fubini--Study distance:}  This figure shows the packing diameters of $N$ points in the Grassmannian $\FGspace{K}{d}$ equipped with the Fubini--Study distance.  The circles indicate the best real packings ($\mathbb{F} = \Rspace{}$) obtained by the authors, while the crosses indicate the best complex packings ($\mathbb{F} = \Cspace{}$) obtained.} \label{fig:fs-packings}

\begin{center}
\includegraphics[width=\textwidth]{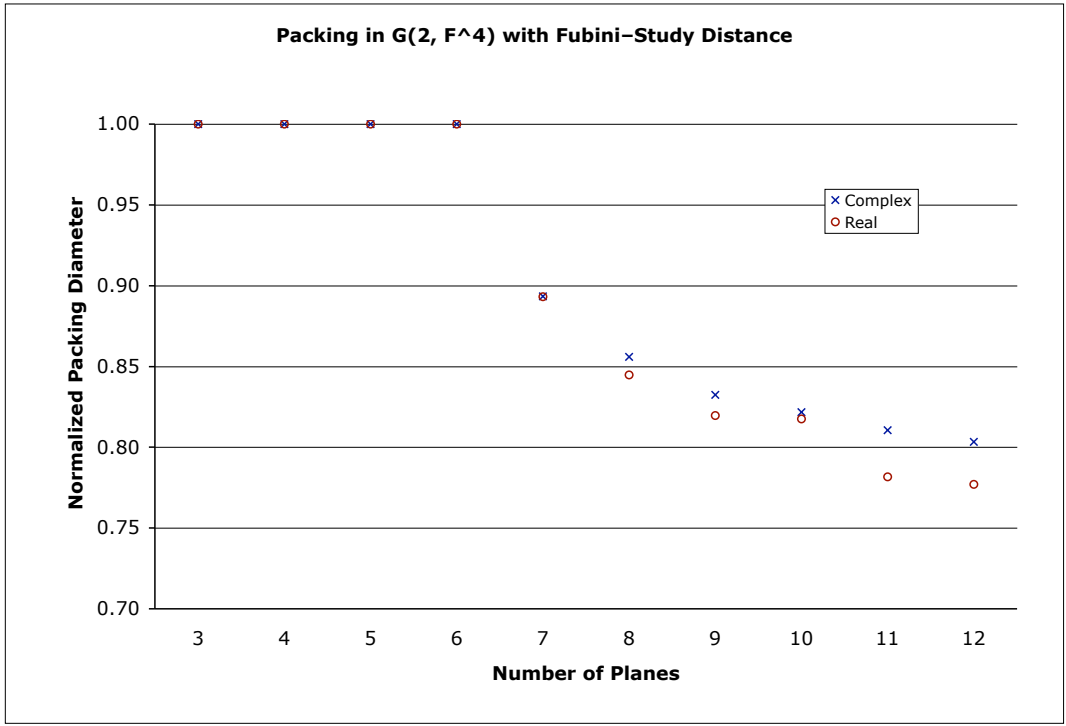} \\
\end{center}
\end{figure}

%%%%%%%%%%%%%%%%%%%%%%%%%%%%%%%%%%%%%%%%%%%%%%%%%%%%%%%%%%%%%%%%%%%%
%%%%%%%%%%%%%%%%%%%%%%%%%%%%%%%%%%%%%%%%%%%%%%%%%%%%%%%%%%%%%%%%%%%%

% Table of spherical packings

%\tablefirsthead{%
%	\cline{3-8}
%	\multicolumn{2}{c|}{} &
%	\multicolumn{5}{|c|}{\textsc{Packing diameters (Degrees)}} &
%	\multicolumn{1}{|c|}{\textsc{Iterations}} \\
%	\hline
%	$d$ & $N$ & NJAS & Best of 10 & \phantom{12} Error &
%	Avg.\ of 10 & \phantom{12} Error & Avg.\ of 10 \\
%%	\multicolumn{1}{|c}{$d$} &
%%	\multicolumn{1}{c|}{$N$} &
%%	\multicolumn{1}{|c|}{NJAS} &
%%	\multicolumn{1}{|c}{Best of 10} &
%%	\multicolumn{1}{c|}{\phantom{12} Error} &
%%	\multicolumn{1}{|c}{Avg.\ of 10} &
%%	\multicolumn{1}{c|}{\phantom{12} Error} &
%%	\multicolumn{1}{|c|}{Avg.\ of 10} \\
%	\hline\hline}

\clearpage

\newpage

\begin{center}
\tablefirsthead{%
	\cline{3-8}
	\multicolumn{2}{c|}{} &
	\multicolumn{5}{|c|}{\textsc{Packing diameters (Degrees)}} &
	\multicolumn{1}{|c|}{\textsc{Iterations}} \\
	\hline
	$d$ & $N$ & NJAS & Best of 10 & \phantom{12} Error &
	Avg.\ of 10 & \phantom{12} Error & Avg.\ of 10 \\
	\hline\hline}
\tablehead{%
	\multicolumn{8}{l}{\textsl{\dots continued}} \\
	\multicolumn{8}{l}{} \\
	\cline{3-8}
	\multicolumn{2}{c|}{} &
	\multicolumn{5}{|c|}{\textsc{Packing diameters (Degrees)}} &
	\multicolumn{1}{|c|}{\textsc{Iterations}} \\
	\hline
	$d$ & $N$ & NJAS & Best of 10 & \phantom{12} Error &
	Avg.\ of 10 & \phantom{12} Error & Avg.\ of 10 \\
	\hline\hline}
\tabletail{%
	\hline
	\multicolumn{8}{r}{} \\
	\multicolumn{8}{r}{\textsl{continued\dots}} \\}
\tablelasttail{\hline}

\topcaption[Packing on spheres]{\textsc{Packing on spheres:}  For collections of $N$ points on the $(d - 1)$-dimensional sphere, this table lists the best packing diameter and the average packing diameter obtained during ten random trials of the alternating projection algorithm.  The error columns record how far our results decline from the putative optimal packings (NJAS) reported in \cite{Slo:Sphere-Packing-Web}.  The last column gives the average number of iterations of alternating projection per trial.} \label{tab:sphere-packing}

\small
\begin{supertabular}{|cr|r|rr|rr|r|}
	3	&	4	&	109.471	&	109.471	&	0.001	&	109.471	&	0.001	&	45	\\
	3	&	5	&	90.000	&	90.000	&	0.000	&	89.999	&	0.001	&	130	\\
	3	&	6	&	90.000	&	90.000	&	0.000	&	90.000	&	0.000	&	41	\\
	3	&	7	&	77.870	&	77.869	&	0.001	&	77.869	&	0.001	&	613	\\
	3	&	8	&	74.858	&	74.858	&	0.001	&	74.858	&	0.001	&	328	\\
	3	&	9	&	70.529	&	70.528	&	0.001	&	70.528	&	0.001	&	814	\\
	3	&	10	&	66.147	&	66.140	&	0.007	&	66.010	&	0.137	&	5000	\\
	3	&	11	&	63.435	&	63.434	&	0.001	&	63.434	&	0.001	&	537	\\
	3	&	12	&	63.435	&	63.434	&	0.001	&	63.434	&	0.001	&	209	\\
	3	&	13	&	57.137	&	57.136	&	0.001	&	56.571	&	0.565	&	4876	\\
	3	&	14	&	55.671	&	55.670	&	0.001	&	55.439	&	0.232	&	3443	\\
	3	&	15	&	53.658	&	53.620	&	0.038	&	53.479	&	0.178	&	5000	\\
	3	&	16	&	52.244	&	52.243	&	0.001	&	51.665	&	0.579	&	4597	\\
	3	&	17	&	51.090	&	51.084	&	0.007	&	51.071	&	0.019	&	5000	\\
	3	&	18	&	49.557	&	49.548	&	0.008	&	49.506	&	0.050	&	5000	\\
	3	&	19	&	47.692	&	47.643	&	0.049	&	47.434	&	0.258	&	5000	\\
	3	&	20	&	47.431	&	47.429	&	0.002	&	47.254	&	0.177	&	5000	\\
	3	&	21	&	45.613	&	45.576	&	0.037	&	45.397	&	0.217	&	5000	\\
	3	&	22	&	44.740	&	44.677	&	0.063	&	44.123	&	0.617	&	5000	\\
	3	&	23	&	43.710	&	43.700	&	0.009	&	43.579	&	0.131	&	5000	\\
	3	&	24	&	43.691	&	43.690	&	0.001	&	43.689	&	0.002	&	3634	\\
	3	&	25	&	41.634	&	41.458	&	0.177	&	41.163	&	0.471	&	5000	\\
\hline\hline
	4	&	5	&	104.478	&	104.478	&	0.000	&	104.267	&	0.211	&	2765	\\
	4	&	6	&	90.000	&	90.000	&	0.000	&	89.999	&	0.001	&	110	\\
	4	&	7	&	90.000	&	89.999	&	0.001	&	89.999	&	0.001	&	483	\\
	4	&	8	&	90.000	&	90.000	&	0.000	&	89.999	&	0.001	&	43	\\
	4	&	9	&	80.676	&	80.596	&	0.081	&	80.565	&	0.111	&	5000	\\
	4	&	10	&	80.406	&	80.405	&	0.001	&	77.974	&	2.432	&	2107	\\
	4	&	11	&	76.679	&	76.678	&	0.001	&	75.881	&	0.798	&	2386	\\
	4	&	12	&	75.522	&	75.522	&	0.001	&	74.775	&	0.748	&	3286	\\
	4	&	13	&	72.104	&	72.103	&	0.001	&	71.965	&	0.139	&	4832	\\
	4	&	14	&	71.366	&	71.240	&	0.126	&	71.184	&	0.182	&	5000	\\
	4	&	15	&	69.452	&	69.450	&	0.002	&	69.374	&	0.078	&	5000	\\
	4	&	16	&	67.193	&	67.095	&	0.098	&	66.265	&	0.928	&	5000	\\
	4	&	17	&	65.653	&	65.652	&	0.001	&	64.821	&	0.832	&	4769	\\
	4	&	18	&	64.987	&	64.987	&	0.001	&	64.400	&	0.587	&	4713	\\
	4	&	19	&	64.262	&	64.261	&	0.001	&	64.226	&	0.036	&	4444	\\
	4	&	20	&	64.262	&	64.261	&	0.001	&	64.254	&	0.008	&	3738	\\
	4	&	21	&	61.876	&	61.864	&	0.012	&	61.570	&	0.306	&	5000	\\
	4	&	22	&	60.140	&	60.084	&	0.055	&	59.655	&	0.485	&	5000	\\
	4	&	23	&	60.000	&	59.999	&	0.001	&	58.582	&	1.418	&	4679	\\
	4	&	24	&	60.000	&	58.209	&	1.791	&	57.253	&	2.747	&	5000	\\
	4	&	25	&	57.499	&	57.075	&	0.424	&	56.871	&	0.628	&	5000	\\
\hline\hline
	5	&	6	&	101.537	&	101.536	&	0.001	&	95.585	&	5.952	&	4056	\\
	5	&	7	&	90.000	&	89.999	&	0.001	&	89.999	&	0.001	&	1540	\\
	5	&	8	&	90.000	&	89.999	&	0.001	&	89.999	&	0.001	&	846	\\
	5	&	9	&	90.000	&	89.999	&	0.001	&	89.999	&	0.001	&	388	\\
	5	&	10	&	90.000	&	90.000	&	0.000	&	89.999	&	0.001	&	44	\\
	5	&	11	&	82.365	&	82.300	&	0.065	&	81.937	&	0.429	&	5000	\\
	5	&	12	&	81.145	&	81.145	&	0.001	&	80.993	&	0.152	&	4695	\\
	5	&	13	&	79.207	&	79.129	&	0.078	&	78.858	&	0.349	&	5000	\\
	5	&	14	&	78.463	&	78.462	&	0.001	&	78.280	&	0.183	&	1541	\\
	5	&	15	&	78.463	&	78.462	&	0.001	&	77.477	&	0.986	&	1763	\\
	5	&	16	&	78.463	&	78.462	&	0.001	&	78.462	&	0.001	&	182	\\
	5	&	17	&	74.307	&	74.307	&	0.001	&	73.862	&	0.446	&	4147	\\
	5	&	18	&	74.008	&	74.007	&	0.001	&	73.363	&	0.645	&	3200	\\
	5	&	19	&	73.033	&	73.016	&	0.017	&	72.444	&	0.589	&	5000	\\
	5	&	20	&	72.579	&	72.579	&	0.001	&	72.476	&	0.104	&	4689	\\
	5	&	21	&	71.644	&	71.639	&	0.005	&	71.606	&	0.039	&	5000	\\
	5	&	22	&	69.207	&	68.683	&	0.524	&	68.026	&	1.181	&	5000	\\
	5	&	23	&	68.298	&	68.148	&	0.150	&	67.568	&	0.731	&	5000	\\
	5	&	24	&	68.023	&	68.018	&	0.006	&	67.127	&	0.896	&	5000	\\
	5	&	25	&	67.690	&	67.607	&	0.083	&	66.434	&	1.256	&	5000	\\
\end{supertabular}

\normalsize
\end{center}

\clearpage

\bibliographystyle{alpha}
\bibliography{packing}

\end{document}